\documentclass[11pt]{amsart}
\usepackage{graphicx}
\usepackage{hyperref}
\usepackage{subfig}

\usepackage{amsthm}
\usepackage[centering,margin=1in]{geometry}

\def \integers{\mathbb{Z}}
\def \reals{\mathbb{R}}

\usepackage[numbers]{natbib}
\usepackage{amsmath,amssymb,enumerate}
\usepackage[all]{xy}
\usepackage{mathtools}
\usepackage{ytableau}
\usepackage{longtable}
\ytableausetup{boxsize=2.4ex,centertableaux}

\theoremstyle{plain} 
\newtheorem{theorem}{Theorem}[section]
\newtheorem{lemma}[theorem]{Lemma}
\newtheorem{proposition}[theorem]{Proposition}
\newtheorem{corollary}[theorem]{Corollary}
\newtheorem{conjecture}[theorem]{Conjecture}

\theoremstyle{definition}
\newtheorem{example}[theorem]{Example}
\newtheorem{definition}[theorem]{Definition}

\theoremstyle{remark}
\newtheorem{remark}[theorem]{Remark}
\newtheorem{remarks}[theorem]{Remarks}

\newlength{\cellsize}
\newcommand\tableau[1]{
\vcenter{
\let\\=\cr
\baselineskip=-16000pt
\lineskiplimit=16000pt
\lineskip=0pt
\halign{&\tableaucell{##}\cr#1\crcr}}}

\newcommand{\tableaucell}[1]{{%
\def \arg{#1}\def \void{}%

\ifx \void \arg
\vbox to \cellsize{\vfil \hrule width \cellsize height 0pt}%
\else
\unitlength=\cellsize
\begin{picture}(1,1)
\put(0,0){\makebox(1,1){$#1$}}
\put(0,0){\line(1,0){1}}
\put(0,1){\line(1,0){1}}
\put(0,0){\line(0,1){1}}
\put(1,0){\line(0,1){1}}
\end{picture}%
\fi}}
\newcommand{\qstep}{
\lhd
}

\def \A{\mathcal{A}}
\def \B{B^{\otimes \lambda} }
\def \R{\mathbb{R}}
\def \Z{\mathbb{Z}}
\def \h{\mathfrak{h}}
\def \g{\mathfrak{g}}

\newcommand{\QB}{\mathrm{QB}}

\def \sgn{\mathrm{sgn}}

\def \g{\mathfrak{g}}
\def \wt{\mathrm{wt}}

\def \height{\mathrm{height} }
\def \htroot{\mathrm{ht}}
\def \fill{\mathrm{fill}}
\def \sfill{\mathrm{sfill}}
\def \Bzero{\mathbf{0}}

\newcommand{\inner}[2]{\langle #1, #2 \rangle}
\renewcommand{\l}[1] {l^{J}_{#1}}

\newcommand{\casetwox}[4]{\left\{ 
\begin{array}{ll}
#1 &\mbox{if $#2$} \\ #3 &\mbox{if $#4$}\,. 
\end{array}
\right.}

\newcommand{\casetwoex}[4]{\left\{ 
\begin{array}{ll}
#1 &\mbox{if $#2$} \\ #3 &\mbox{if $#4$} \,. 
\end{array}
\right.}

\begin{document}
\title{A uniform realization of the combinatorial $R$-matrix}
\author{Cristian Lenart} 
\address{Department of Mathematics and Statistics, State University of New York at Albany,
Albany, NY 12222, USA}
\email{clenart@albany.edu}
\author{Arthur Lubovsky}
\address{Department of Mathematics and Statistics, State University of New York at Albany,
Albany, NY 12222, USA}
\email{alubovsky@albany.edu}
\thanks{C.L. was partially supported by the NSF grants DMS--1101264 and DMS--1362627, and gratefully acknowledges the hospitality and support of the Max-Planck-Institut f\"ur Mathematik in Bonn, where part of this work was carried out. A.L. was supported by a GAANN grant from the US Department of Education.}

\begin{abstract}

    Kirillov-Reshetikhin crystals are colored directed graphs encoding the structure of certain finite-dimensional representations of affine Lie algebras. A tensor products of column
shape Kirillov-Reshetikhin crystals has recently been realized in a uniform way, for all untwisted affine types, in terms of the quantum alcove model.
We enhance this model by using it to give a uniform realization of the {combinatorial $R$-matrix}, i.e., the unique affine crystal isomorphism 
permuting factors in a tensor product of KR crystals. In other words, we are generalizing to all Lie types Sch\"utzenberger's sliding game (jeu de taquin) for Young tableaux, which  realizes the combinatorial $R$-matrix in type $A$. Our construction is in terms of certain combinatorial moves, called quantum Yang-Baxter moves, which are explicitly described by reduction to the rank 2 root systems. We also show that the quantum alcove model does not depend on the choice of a sequence of alcoves joining the fundamental one to a translation of it.

\end{abstract}
\maketitle

\section{Introduction}
Kashiwara's {\em crystals} \cite{kascbq} are colored directed graphs encoding the structure of certain
bases (called crystal bases) of some representations of quantum groups
$U_q({\mathfrak g})$ as $q$ goes to zero (where ${\mathfrak g}$ is a symmetrizable Kac-Moody Lie algebra). 
All highest weight representations have crystal bases/graphs. Beside them, an important class of crystals 
is represented by the {\em Kirillov-Reshetikhin (KR) crystals} \cite{karrym}. They correspond to certain finite-dimensional modules  for affine Lie algebras which are not of highest weight. A KR crystal is denoted $B^{r,s}$, being labeled by an $r\times s$ rectangle, where the height $r$ indexes a simple root of the corresponding finite root system and the width $s$ is any positive integer. The importance of KR crystals stems from the fact that they are building blocks for the corresponding (infinite) highest weight crystals; indeed, the latter are realized as infinite tensor products of the former in the Kyoto path model, see, e.g., \cite{hkqgcb}. Tensor products of KR crystals are endowed with a grading known as the {\em energy function} \cite{NS08,satdck}, which originates in the theory of solvable lattice models \cite{hkorff}. There is a unique affine crystal isomorphism between two tensor products of KR crystals differing by a permutation of the tensor factors; it is called the {\em combinatorial $R$-matrix}.

The first author and Postnikov \cite{lapawg,lapcmc} defined the so-called
{\em alcove model} for highest weight crystals associated to a
 symmetrizable Kac-Moody algebra. A related model is the one of
Gaussent-Littelmann, based on LS-galleries \cite{gallsg}. Both models
are discrete counterparts of the celebrated {\em Littelmann path model} \cite{litlrr,litpro}.  In
\cite{lalgam} the authors generalize the alcove model.  This
generalization, called the {\em quantum alcove model}, has been shown in
\cite{unialcmod2} to uniformly describe tensor products of column
shape KR crystals for all
untwisted affine types.  By contrast, all the existing combinatorial models for KR crystals are type-specific; most of them correspond to the classical types, and are based on diagram fillings, i.e., on tableau models \cite{foskr}. As far as the energy function is concerned, in the quantum alcove model it is computed uniformly and efficiently by a statistic called height \cite{unialcmod2}, whereas an efficient computation based on the tableau models is only available in types $A$ and $C$ \cite{lascec}.

In this paper we enhance the quantum alcove model by using it to give a uniform realization of the {combinatorial $R$-matrix}. The construction is based on certain combinatorial moves called {\em quantum Yang-Baxter moves}, which generalize their alcove model versions defined in \cite{lenccg}. These moves are explicitly described in all Lie types by reduction to the rank 2 root systems. Note that, as far as existing realizations of the combinatorial $R$-matrix are concerned, they are limited in scope and type-specific. For instance, in terms of the tableau models, there is a construction in type $A$ based on Sch\"utzenberger's {\em jeu de taquin} (sliding algorithm) on two columns \cite{fulyt}, whereas the extensions of this procedure to types $B$ and $C$ are involved and not transparent, see \cite{lecsc,lecst}. By contrast, our construction is easy to formulate, and is related to more general concepts. 

We also show that, like the alcove model, its quantum generalization does not depend on the choice of a sequence of roots called a $\lambda$-chain (or, equivalently, on the choice of a sequence of alcoves joining the fundamental one to a translation of it). Note that the similar statement for the Littelmann path model was proved in \cite{litpro} based on subtle continuous arguments, whereas the alcove model and its quantum generalization have the advantage of being discrete, so they are amenable to the use of the combinatorial methods mentioned above.

\section{Background}
\subsection{Root systems}\label{notroot}

Let $\mathfrak{g}$ be a complex simple Lie algebra, and $\mathfrak{h}$ a Cartan subalgebra, whose rank is $r$. 
Let $\Phi \subset \mathfrak{h}^*$ be the corresponding irreducible \emph{root system}, 
$\mathfrak{h}^*_{\reals}\subset \mathfrak{h}$ the real span of the roots, and $\Phi^{+} \subset \Phi$ the set of positive roots.
Let $\Phi^{-} := \Phi \backslash \Phi^{+}$.
For $ \alpha \in \Phi$, we say that $ \alpha > 0 $ if $ \alpha \in \Phi^{+}$,
and $ \alpha < 0 $ if $ \alpha \in \Phi^{-}$.
The sign of the root $\alpha$, denoted $\sgn(\alpha)$, is defined to be $1$ if $\alpha \in \Phi^{+}$, and $-1$ otherwise.
Let $| \alpha | = \sgn( \alpha ) \alpha $.
Let $\rho := \frac{1}{2}(\sum_{\alpha \in \Phi^{+}}\alpha)$. We denote, as usual, the reflection corresponding to the root $\alpha$ by $s_\alpha$. Let 
$\alpha_1, \ldots , \alpha_r \in \Phi^{+}$ be the \emph{simple roots}, and $s_i:=s_{\alpha_i}$ the corresponding simple reflections; the latter generate the {\em Weyl group} $W$. We denote 
$\inner{\cdot}{\cdot}$ the non-degenerate scalar product on 
$\mathfrak{h}^{*}_{\reals}$ induced by the Killing form. Given a root $\alpha$, we consider the corresponding \emph{coroot} $\alpha^{\vee} := 2\alpha/ \inner{\alpha}{\alpha}$.
If $\alpha= \sum_i c_i \alpha_i$, then the \emph{height} of $\alpha$,
denoted by $\htroot(\alpha)$, is given by 
$\htroot(\alpha):=\sum_i c_i$. 
We denote by $\widetilde{\alpha}$ the highest root in $\Phi^{+}$; we let 
$\theta = \alpha_0 := -\widetilde{\alpha} $ and $s_0:=s_{\widetilde{\alpha}}$.

The \emph{weight lattice} $\Lambda$ is given by 
\begin{equation}
	\Lambda := \left\{ \lambda \in \mathfrak{h}_{\reals}^{*} \, : \,
	\inner{\lambda}{\alpha^{\vee}} \in \integers \text{ for any } \alpha \in \Phi \right\}.
	\label{eqn:weight_lattice}
\end{equation}
The weight lattice $\Lambda$ is generated by the \emph{fundamental weights} 
$\omega_1, \ldots \omega_r$, which form the dual basis to the basis of simple coroots, i.e., 
$\inner{\omega_i}{\alpha_j^{\vee}}= \delta_{ij}$. The set $\Lambda^{+}$ of \emph{dominant weights}
is given by 
\begin{equation}
	\Lambda^{+} := \left\{  \lambda \in \Lambda \, : \, 
	\inner{\lambda}{\alpha^{\vee}} \geq 0 \text{ for any } \alpha \in \Phi^{+}
	\right\}.
	\label{eqn:dominant_weights}
\end{equation}

Given $\alpha \in \Phi$ and $k \in \integers$, we denote by $s_{\alpha,k}$ the reflection in the affine hyperplane
\begin{equation}
	H_{\alpha,k}:= \left\{  \lambda \in \mathfrak{h}^{*}_{\reals} \, : \, 
	\inner{\lambda}{\alpha^{\vee}} = k \right\}
	\label{eqn:affine_hyperplane}.
\end{equation}
These reflections generate the \emph{affine Weyl group} $W_{\textrm{aff}}$ for the 
\emph{dual root system} $\Phi^{\vee}:= \left\{ \alpha^{\vee} \, |\, \alpha \in \Phi \right\}$.
The hyperplanes $H_{\alpha,k}$ divide the real vector space $\mathfrak{h}^{*}_\reals$ into open regions, called \emph{alcoves.} The \emph{fundamental alcove} $A_{\circ}$ is given by
\begin{equation}
	A_{\circ} := \left\{ \lambda \in \mathfrak{h}_{\reals}^{*} \, | \, 
	0 < \inner{\lambda}{\alpha^{\vee}} < 1 \text{ for all } \alpha \in \Phi^{+}
	\right\}.
	\label{eqn:fundamental_alcove}
\end{equation}

\subsection{Weyl groups} Let $W$ be the {Weyl group} of the root system $\Phi$ discussed above. The length function on $W$ is denoted by $\ell(\cdot)$. The \emph{Bruhat order} on $W$ is defined by its covers $w \lessdot ws_{\alpha}$, for $\alpha \in \Phi^{+}$,
if $\ell(ws_{\alpha}) = \ell(w) + 1$. 
Define 
\begin{equation}\label{downqbg}w \qstep ws_{\alpha}\,,\;\mbox{ for $\alpha \in \Phi^{+}$}\,,\;\mbox{ if 
%$\ell(ws_{\alpha}) =  \ell(w) - 2\inner{\rho}{\alpha^{\vee}} + 1$.
$\ell(ws_{\alpha}) =  \ell(w) - 2\htroot( \alpha^\vee )  + 1$}\,.\end{equation}
The \emph{quantum Bruhat graph} \cite{fawqps} is the directed graph on $W$ with edges labeled by positive roots
\begin{equation}
	w \stackrel{\alpha}{\longrightarrow} 
	ws_{\alpha} \quad 
	\text{ for  } 
	w \lessdot ws_{\alpha} \,\mbox{ or }\, w \qstep ws_{\alpha}.
	\label{eqn:qbruhat_edge}
\end{equation}
We denote this graph by $\QB(W)$.

We recall an important topological property of $\QB(W)$, called {\em shellability}, which was proved in \cite{bfpmbo}. This is defined with respect to a {\em reflection ordering} on the positive roots \cite{dyehas}. 

\begin{theorem}{\rm \cite{bfpmbo}}\label{thm:shell} 
Fix a reflection ordering on $\Phi^+$.
\begin{enumerate}
\item[{\rm (1)}] For any pair of elements $v,w\in W$, there is a unique path from $v$ to $w$ in the quantum 
Bruhat graph $\QB(W)$ such that its sequence of edge labels is strictly increasing (resp., decreasing) 
with respect to the reflection ordering.
\item[{\rm (2)}] The path in {\rm (1)} has the smallest possible length and is lexicographically minimal 
(resp., maximal) among all shortest paths from $v$ to $w$.
\end{enumerate}
\end{theorem}

\subsection{Kirillov-Reshetikhin (KR) crystals}
\label{subsection:KR-crystals}
Given a symmetrizable Kac-Moody algebra $\g$, a $\g$-{\em crystal} is a non-empty set $B$ together with maps $e_i,f_i:B\to B\cup 
\{ \Bzero \}$ for $i\in I$ (where $I$ indexes the simple roots corresponding to $\g$, as usual, and $\Bzero \not \in B$), and 
$\wt:B \to \Lambda$. We require $b'=f_i(b)$ if and only if $b=e_i(b')$, and $\wt(f_i(b))=\wt(b)-\alpha_i$. The maps $e_i$ and 
$f_i$ are called {\em crystal operators} 
and are represented as arrows $b \to b'=f_i(b)$ colored $i$; thus they endow $B$ with the structure of a colored directed graph.
For $b\in B$, we set $\varepsilon_i(b) := \max\{k \mid e_i^k(b) \neq \Bzero \}$, and 
$\varphi_i(b) := \max\{k \mid f_i^k(b) \neq \Bzero \}$.
Given two $\g$-crystals $B_1$ and $B_2$, we define their tensor product $B_1 \otimes B_2$ 
as follows.
As a set, $B_1\otimes B_2$ is the Cartesian
product of the two sets. For $b=b_1 \otimes b_2\in B_1 \otimes B_2$, the weight function is simply
$\wt(b) := \wt(b_1) + \wt(b_2)$. 
The crystal operators are given by
\begin{equation}\label{tens1}
	f_i (b_1 \otimes b_2):=
	\begin{cases}
	f_i  (b_1) \otimes b_2 & \text{if $\varepsilon_i(b_1) \geq \varphi_i(b_2)$}\\
    b_1 \otimes f_i (b_2)  & \text{ otherwise, }
	\end{cases}
\end{equation}
and similarly for $e_i$. 
The {\em highest weight crystal} $B(\lambda)$ of highest weight $\lambda\in \Lambda^+$ is a 
certain crystal with a unique element 
$u_\lambda$ such that $e_i(u_\lambda)=\Bzero$ for all $i\in I$ 
and $\wt(u_\lambda)=\lambda$.
It encodes the structure of the crystal basis of the $U_q(\mathfrak{g})$-irreducible 
representation with highest weight $\lambda$ as $q$ goes to 0.

A {\em Kirillov-Reshetikhin (KR) crystal} \cite{karrym} is a finite crystal $B^{r,s}$ for an affine algebra, labeled by a rectangle of height $r$ and width $s$, where $r\in I\setminus\{0\}$ and $s$ is any positive integer. We refer, throughout, to the untwisted affine types $A_{n-1}^{(1)}-G_2^{(1)}$, and only consider column shape KR crystals $B^{r,1}$. 

As an example, consider the KR crystal $B^{r,1}$ of type $A_{n-1}^{(1)}$ 
%and $C_n^{(1)}$, 
with $r\in \{1,2,\ldots,n-1\}$, for which we have a simple tableau model. %and $r\in \{1,2,\ldots,n\}$, respectively. 
As a classical type $A_{n-1}$ 
%(resp. $C^{(1)}_n$) 
crystal, 
$B^{r,1}$ is isomorphic to the corresponding crystal $B(\omega_r)$.
Recall that an element $b\in B(\omega_r)$ is 
represented by a strictly increasing filling 
of a height $r$ column, with entries in $[n]:=\{1, \dots , n \}$. 
There is a simple construction of the crystal operators on a tensor product of (column shape) KR crystals of type $A_{n-1}^{(1)}$, which is based on \eqref{tens1}.

We refer again to (column shape) KR crystals of arbitrary (untwisted) type. 
Let $\lambda=(\lambda_1  \geq \lambda_2\geq \ldots)$ be a partition, and $\lambda'$ the conjugate partition.
We define $\B:=\bigotimes_{i=1}^{\lambda_1} B^{\lambda_i',1}$. More generally, given a composition ${\mathbf{p}}=(p_1,\ldots,p_k)$, we define
$B^{\otimes {\mathbf{p}}}:=\bigotimes_{i=1}^{k} B^{p_i,1}$. (In both cases, we assume that the corresponding column shape KR crystals exist.) We denote such a tensor product generically by $B$. 

\begin{remarks}\label{rem:conn} {\rm (1)} It is known that $B$ is connected as an affine crystal, but disconnected as a classical crystal (i.e., with the $0$-arrows removed). 

{\rm (2)} Let ${\mathbf{p}'}$ be a composition obtained from ${\mathbf{p}}$ by permuting its parts. There is an affine crystal isomorphism between $B^{\otimes {\mathbf{p}}}$ and $B^{\otimes {\mathbf{p'}}}$, which is unique by the previous remark. This isomorphism is called the {\em combinatorial $R$-matrix}.  
\end{remarks}

We need to distinguish certain arrows in $B$, which are related to affine Demazure crystals, as we shall explain.

\begin{definition}%[\cite{demazure_arrows, KRcrystals_energy}]
An arrow $b\rightarrow f_i(b)$ in $B$ is called a \emph{Demazure arrow} if $i \ne 0$, or $i=0$ and 
$\varepsilon_0(b)\ge 1$. 
An arrow $b\rightarrow f_i(b)$ in $B$ is called a \emph{dual Demazure arrow} if $i \ne 0$, or $i=0$ and 
$\varphi_i(b)\geq 2$. 
\end{definition}

\begin{remarks}\label{rem:demazure} {\rm (1)} By Fourier-Littelmann \cite{faltps}, in simply-laced types, the tensor product of KR crystals $B$ is isomorphic, as a classical crystal (discard the affine $0$-arrows) with a certain {\em Demazure crystal} for the corresponding affine algebra. (Demazure modules are submodules of highest weight ones determined by a Borel subalgebra acting on an extremal weight vector.)  Moreover, by \cite{fssdsi}, the $0$-arrows in the latter correspond precisely to the Demazure arrows in $B$. 

{\rm (2)} In the case when all of the tensor factors in $B$ are \emph{perfect}
crystals \citep{hkqgcb}, $B$ remains connected upon removal of the non-Demazure (resp. non-dual Demazure) $0$-arrows. 

{\rm (3)} In classical types, $B^{k,1}$ is perfect as follows: in types $A^{(1)}_{n-1}$ and $D^{(1)}_{n}$ for all $k$, in type $B^{(1)}_{n}$ only for $k\ne n$, and in type $C^{(1)}_n$ only for $k=n$ (using the standard indexing of the Dynkin diagram); in other words, for all the Dynkin nodes in simply-laced types, and only for the nodes corresponding to the long roots in non-simply-laced types. It was conjectured in \cite{hkorff} that the same is true in the exceptional types. In type $G_2^{(1)}$ this was confirmed in \cite{yampfg}, while for types $E_{6,7,8}^{(1)}$ (except for two Dynkin nodes for type $E_8^{(1)}$) and $F_{4}^{(1)}$ it was checked by computer, based on a model closely related to the quantum alcove model, see Section \ref{subsection:qalcove} and \cite{unialcmod2}.
\end{remarks}

The \emph{energy function} $D=D_B$ 
is a function from $B$ to the integers, defined 
by summing the so-called local 
energies of all pairs of tensor factors \cite{hkorff}. We will only refer here to the so-called tail energy \cite{unialcmod2}, so we will not make this specification. (There are two conventions in defining the local energy of a pair of tensor factors: commuting the right one towards the head of the tensor product, or the left one towards the tail; the tail energy corresponds to the second choice.) 
We will only need the following property of the energy function, which exhibits it as an affine grading on $B$.

\begin{theorem}{\rm \cite{NS08,satdck}}
	\label{theorem:energy_recursion}
The energy is preserved by the classical crystal operators $f_i$, i.e., $i\ne 0$.
	If $b\rightarrow f_0(b)$ is a dual Demazure arrow, then $D(f_0(b))=D(b)-1$.
\end{theorem}

\begin{remark} Theorem \ref{theorem:energy_recursion} shows that the energy is determined up to a constant on the  
connected components of the 
subgraph of the 
affine crystal $B$ containing only the dual Demazure arrows. See also Remark \ref{rem:demazure} (2). 
\end{remark}

\subsection{The quantum alcove model}
\label{subsection:qalcove}

In this section we recall the quantum alcove model, which is a model for KR crystals corresponding to a fixed untwisted affine Lie algebra $\widehat{\g}$. This model is based on the combinatorics of the root system of the corresponding finite-dimensional Lie algebra $\g$, so we use freely the notation in Section \ref{notroot}. 

We say that two alcoves are \emph{adjacent} if they are distinct and have a common wall. Given a pair of adjacent alcoves $A$ and $B$, we write $A \stackrel{\beta}{\longrightarrow} B$ if the
common wall is of the form $H_{\beta,k}$ and the root
$\beta \in \Phi$ points in the direction from $A$ to $B$.

\begin{definition}%[\cite{lapawg}]
	An  \emph{alcove path} is a {sequence of alcoves} $(A_0, A_1, \ldots, A_m)$ such that
	$A_{j-1}$ and $A_j$ are adjacent, for $j=1,\ldots, m.$ We say that an alcove path 
	is \emph{reduced} if it has minimal length among all alcove paths from $A_0$ to $A_m$.
	\end{definition}
	
	Let $A_{\lambda}=A_{\circ}+\lambda$ be the translation of the fundamental alcove $A_{\circ}$ by the weight $\lambda$.
	
	\begin{definition}%[\cite{lapawg}]
		The sequence of roots $(\beta_1, \beta_2, \dots, \beta_m)$ is called a
		\emph{$\lambda$-chain} if 
		\begin{equation}\label{alcovepath}
		A_0=A_{\circ} \stackrel{-\beta_1}{\longrightarrow} A_1
		\stackrel{-\beta_2}{\longrightarrow}\dots 
		\stackrel{-\beta_m}{\longrightarrow} A_m=A_{-\lambda}\end{equation}
is a reduced alcove path.
\end{definition}

We now fix a dominant weight $\lambda$ and an alcove path $\Pi=(A_0, \dots , A_m)$ from 
$A_0 = A_{\circ}$ to $A_m = A_{-\lambda}$. Note that $\Pi$ is determined by the corresponding $\lambda$-chain $\Gamma:=(\beta_1, \dots, \beta_m)$, which consists of positive roots. 
A specific choice of a $\lambda$-chain, called a lex $\lambda$-chain and denoted $\Gamma_{\rm lex}$, is given in \cite{lapcmc}[Proposition~4.2]; this choice depends on a total order on the simple roots. 
We let $r_i:=s_{\beta_i}$, and let $\widehat{r_i}$ be the affine reflection in the hyperplane containing the common face of $A_{i-1}$ and $A_i$, for $i=1, \ldots, m$; in other words, 
$\widehat{r}_i:= s_{\beta_i,-l_i}$, where $l_i:=|\left\{ j<i \, ; \, \beta_j = \beta_i \right\} |$. 
We define $\widetilde{l}_i:= \inner{\lambda}{\beta_i^{\vee}} -l_i = |\left\{ j \geq i \, ; \, \beta_j = \beta_i \right\} |$.

Let $J=\left\{ j_1 < j_2 < \cdots < j_s \right\} \subseteq [m]$ and 
define 
$\Gamma(J):=\left( \gamma_1,\gamma_2, \dots, \gamma_m
\right)$, where 
\begin{equation}\label{defgamma}\gamma_k:=r_{j_1}r_{j_2}\dots r_{j_p}(\beta_k)\,,\end{equation}
 with
$j_p$ the largest folding position less than $k$. Then 
\begin{equation}\label{deflev}
H_{|\gamma_k|,-\l{k}}=\widehat{r}_{j_1}\widehat{r}_{j_2}\dots \widehat{r}_{j_p}(H_{\beta_k,-l_k})\,,
\end{equation}
for some $\l{k}$, which is defined by this relation.
Define $\gamma_{\infty} := r_{j_1}r_{j_2}\dots r_{j_s}(\rho)$.  

     The elements of $J$ are called \emph{folding positions}.  
	 Given $i \in J$, we  say that $i$ is a \emph{positive folding position} if 
	 $\gamma_i>0$, and a \emph{negative folding position} if $\gamma_i<0$. 
	 We denote the positive folding positions by $J^{+}$, and the negative 
	 ones by $J^{-}$.
	 We call 
\begin{equation}\label{def:weight}
\mu=\mu(J):=-\widehat{r}_{j_1}\widehat{r}_{j_2}\ldots \widehat{r}_{j_s}(-\lambda)
\end{equation}
 the \emph{weight} of $J$.
We define 
\begin{equation}
	\label{eqn:height_statistic}
	\height(J):= \sum_{j \in J^{-}} \widetilde{l}_j.
\end{equation}
\begin{definition}
	A subset $J=\left\{ j_1 < j_2 < \cdots < j_s \right\} \subseteq [m]$ (possibly empty)
 is \emph{admissible} if
we have the following path in $\QB(W)$:
\begin{equation}
	\label{eqn:admissible}
	1 \stackrel{\beta_{j_1}}{\longrightarrow} r_{j_1} \stackrel{\beta_{j_2}}{\longrightarrow} r_{j_1}r_{j_2} 
	\stackrel{\beta_{j_3}}{\longrightarrow} \cdots \stackrel{\beta_{j_s}}{\longrightarrow} r_{j_1}r_{j_2}\cdots r_{j_s}\,.
\end{equation}
We call $\Gamma(J)$ an \emph{admissible folding}.
We let $\A( \Gamma )$ be the collection of admissible subsets corresponding to the $\lambda$-chain $\Gamma$.
When $\Gamma $ is clear from the context, we may use the notation 
$\A( \lambda )$ instead.
\end{definition}
%
% crystal operators defined next 
%

\begin{remark}\label{spec}
If we restrict to admissible subsets for which the path \eqref{eqn:admissible} has no down steps, we recover the classical alcove model in \cite{lapawg,lapcmc}.
\end{remark}

Next we define combinatorial crystal operators $f_p$ and $e_p$ (where $p\in\{0,\ldots,r\}$ indexes the simple roots corresponding to $\widehat{\g}$) on $\A( \Gamma )$. 
Given $J\subseteq [m]$, not necessarily admissible,  and $\alpha\in \Phi$, we will use the following notation:
	\begin{equation}\label{eqn:ILalpha}
	 I_\alpha = I_{\alpha}(J):= \left\{ i \in [m] \, | \, \gamma_i = \pm \alpha \right\}\,, \qquad %L_{\alpha}=L_{\alpha}(\Delta) := \left\{ \l{i} \, | \, i \in I_{\alpha} \right\}, \\
\widehat{I}_\alpha = \widehat{I}_{\alpha}(J):= I_{\alpha} \cup \{\infty\}\,, %\quad \widehat{L}_{\alpha} = \widehat{L}_{\alpha}(\Delta) := L_{\alpha} \cup \{l_\alpha^{\infty} \},
\end{equation}
and $l_{\alpha}^{\infty}:=\inner{\mu(J)}{\sgn(\alpha)\alpha^{\vee}}$.
The following graphical representation of the heights $l_i^J$ for $i\in{I}_\alpha$ and $l_{\alpha}^{\infty}$  %$\widehat{L}_{\alpha}$ 
is useful for defining the crystal operators. 
Let 
\[\widehat{I}_{\alpha}= \left\{ i_1 < i_2 < \dots < i_n <i_{n+1}=\infty \right\}\,
	\text{ and  }
	\varepsilon_i := 
	\begin{cases}
		\,\,\,\, 1 &\text{ if } i \not \in J\\
		-1 & \text { if } i \in J
	\end{cases}.\,
	\]
If $\alpha > 0$, we define the continuous piecewise linear function 
$g_{\alpha}:[0,n+\frac{1}{2}] \to \reals$ by
\begin{equation}
	\label{eqn:piecewise-linear_graph}
	g_\alpha(0)= -\frac{1}{2}, \;\;\; g'_{\alpha}(x)=
	\begin{cases}
		\sgn(\gamma_{i_k}) & \text{ if } x \in (k-1,k-\frac{1}{2}),\, k = 1, \ldots, n\\
		\varepsilon_{i_k}\sgn(\gamma_{i_k}) & 
		\text{ if } x \in (k-\frac{1}{2},k),\, k=1,\ldots,n \\
		\sgn(\inner{\gamma_{\infty}}{\alpha^{\vee}}) &
		\text{ if } x \in (n,n+\frac{1}{2}).
	\end{cases}
\end{equation}
If $\alpha<0$, we define $g_{\alpha}$ to be the graph obtained by reflecting $g_{-\alpha}$ in
the $x$-axis.
By \cite{lapcmc}, we have 
\begin{equation}
	\label{eqn:graph_height}
	\sgn(\alpha)\l{i_k}=g_\alpha\left(k-\frac{1}{2}\right), k=1, \dots, n, \, 
	\text{ and }\, 
	\sgn(\alpha)l_{\alpha}^{\infty}:=
	\inner{\mu(J)}{\alpha^{\vee}} = g_{\alpha}\left(n+\frac{1}{2}\right).
\end{equation}
\begin{example}
    Suppose $ \alpha >0 $ and the sequence
    $\{( \gamma_i,\, \varepsilon_i)\}$ for $i \in I_{ \alpha }$ is
    \[ 
( \alpha, -1),\, (- \alpha, 1),\, ( \alpha, 1),\, ( \alpha, 1),\, 
( \alpha, -1),\, (- \alpha, 1),\, ( \alpha, -1),\,( \alpha, 1),\,
    \]
    in this order; also assume that 
    $\sgn(\inner{ \gamma_{ \infty}}{ \alpha^{\vee} })= 1$.
    The graph of $ g_{ \alpha } $ is shown in Figure \ref{fig:galpha}.
\end{example}
\begin{figure}[h]
    \includegraphics[scale=.45]{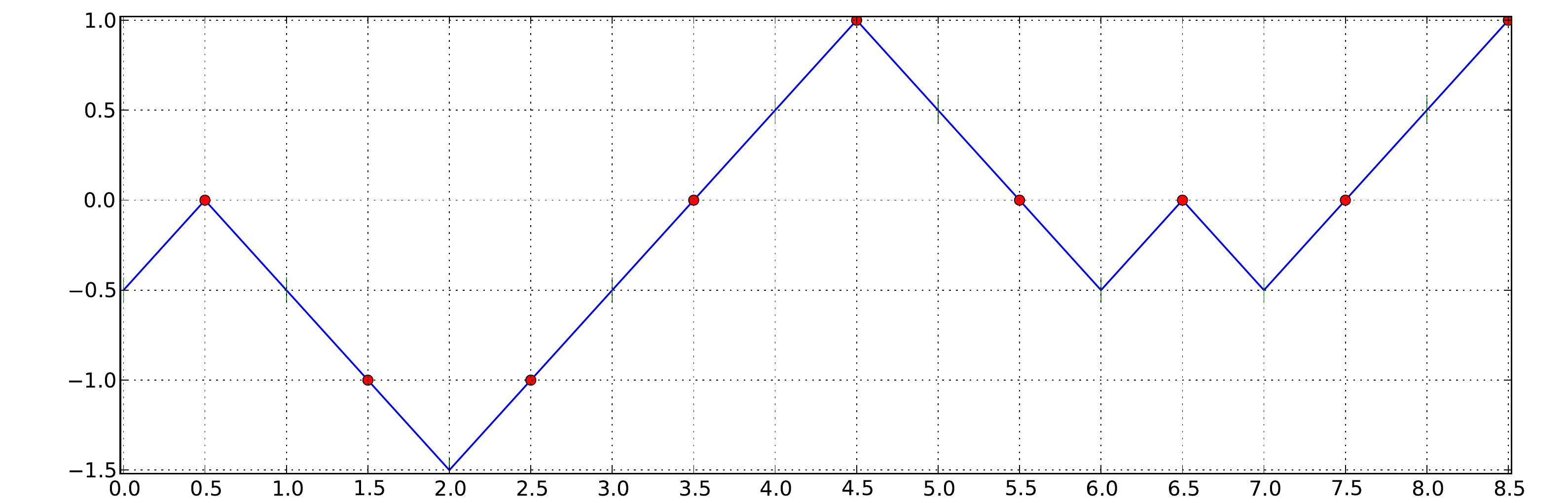}
    \caption{}
    \label{fig:galpha}
\end{figure}

We now impose the additional requirement that $J\subset [m]$ be an admissible subset. 
Let $\delta_{i,j}$ be the Kronecker delta function. 
Fix $p$ in $\{0,\ldots,r\}$, so $\alpha_p$ is a simple root if $p>0$, or $\theta$ if $p=0$.
Let $M$ be the maximum of $g_{{\alpha}_p}$. Let $m$ be the minimum index 
$i$  in $\widehat{I}_{{\alpha}_p}$ for which we have $\sgn({\alpha}_p)\l{i}=M$. 
It was proved in \cite{lalgam} that, if $M\ge\delta_{p,0}$, then either $m\in J$ or $m=\infty$; furthermore, if $M>\delta_{p,0}$, then $m$ has a predecessor $k$ in $\widehat{I}_{{\alpha}_p}$, and we have $k\not\in J$. We define
\begin{equation}
	\label{eqn:rootF} 
f_p(J):= 
	\begin{cases}
		(J \backslash \left\{ m \right\}) \cup \{ k \} & \text{ if $M>\delta_{p,0} $ } \\
				\Bzero & \text{ otherwise }.
	\end{cases}
\end{equation}
Now we define $e_p$. Again let $M:= \max g_{{\alpha}_p}$. Assuming that $M>\inner{\mu(J)}{{\alpha}_p^{\vee}}$, let $k$ be the maximum index 
$i$  in $I_{{\alpha}_p}$ for which we have $\sgn({\alpha}_p)\l{i}=M$,
and let $m$ be the successor of $k$ in $\widehat{I}_{{\alpha}_p}$. Assuming also that $M\ge\delta_{p,0}$, it was proved in \cite{lalgam} that $k\in J$, and either $m\not\in J$ or $m=\infty$. 
Define 
\begin{equation}
	\label{eqn:rootE}
e_p(J):= 
	\begin{cases}
		(J \backslash \left\{ k \right\}) \cup \{ m \} & \text{ if }
		M>\inner{\mu(J)}{{\alpha}_p^{\vee}} \text{ and } M \geq \delta_{p,0}  \\
				\Bzero & \text{ otherwise. }
	\end{cases}
\end{equation}
In the above definitions,  we use the
convention that $J\backslash \left\{ \infty \right\}= J \cup \left\{ \infty \right\} = J$. 

\begin{theorem}{\rm \cite{lalgam}}
    \label{theorem:admissible}
      $\!\!\!\!\!\!\!\!\!\!\!$ \begin{enumerate}
    \item[{\rm (1)}]  If $J$ is an admissible subset and if $f_p(J) \ne \Bzero$, then $f_p(J)$ is also an admissible subset. Similarly for $e_p(J)$. Moreover, $f_p(J)=J'$ if and only if  $e_p(J')=J$.
    \item[{\rm (2)}] We have $ \mu(f_p(J)) = \mu(J) - \alpha_p $. Moreover, if $M\ge\delta_{p,0}$, then
    \[\varphi_p(J)=M-\delta_{p,0}\,,\;\;\;\;\varepsilon_p(J)=M-\langle\mu(J),\alpha_p^\vee\rangle\,,\]
    while otherwise $\varphi_p(J)=\varepsilon_p(J)=0$. 
    \end{enumerate}
\end{theorem}

\begin{remark}\label{changepath}
Let $J=\{j_1 < \dots < j_s \}$ be an admissible subset, and $w_i := r_{j_1}r_{j_2} \dots r_{j_i}$. Let $M,\,m,\,k$ be 
 as in the above definition of $f_p(J)$, assuming $M>\delta_{p,0}$. Now assume that $m\ne\infty$, and  let $a<b$ be such that 
\[
	 j_a < k < j_{a+1} < \dots < j_b = m < j_{b+1} \;; 
\] 
if $a=0$ or $b+1>s$, then the corresponding indices $j_a$, respectively $j_{b+1}$, are missing.
In the proof of Theorem~\ref{theorem:admissible} in \cite{lalgam}, it was shown that $f_p$ has the effect of changing 
the path in the quantum Bruhat graph 
\[
1=w_0\rightarrow\ldots\rightarrow w_a\rightarrow w_{a+1}\rightarrow\ldots\rightarrow w_{b-1}\rightarrow 
w_b\rightarrow\ldots\rightarrow w_s
\]
corresponding to $J$ into the following path corresponding to $f_p(J)$:
\[
1=w_0\rightarrow \ldots\rightarrow w_a\rightarrow s_p w_a\rightarrow s_p w_{a+1}\rightarrow\ldots\rightarrow s_p w_{b-1}= w_b\rightarrow\ldots\rightarrow w_s\,.
\]
The case $m=\infty$ is similar.
\end{remark}

We summarize the results in \cite{unialcmod2}, cf. also \cite{unialcmod, lnseda}, related to the applications of the quantum alcove model.

\begin{theorem}{\rm \cite{unialcmod2}}
    \label{mainconj} Consider a composition ${\mathbf{p}}=(p_1,\ldots,p_k)$ and the corresponding KR crystal $B:=\bigotimes_{i=1}^{k} B^{p_i,1}$. Let $\lambda:=\omega_{p_1}+\ldots+\omega_{p_k}$, and let $\Gamma_{\rm lex}$ be a corresponding lex $\lambda$-chain (see above). 
    
        {\rm (1)}  The (combinatorial) crystal $\A(\Gamma_{\rm lex})$ is isomorphic to the subgraph of $B$ consisting of the dual Demazure arrows, via a specific bijection  which preserves the weights of the vertices. 

        {\rm (2)}  If the vertex $b$ of $B$ corresponds to $J$ under the isomorphism in part {\rm (1)},  then 
         the energy is given by $D_B(b)-C=-\height(J)$, where $C$ is a global constant.
  \end{theorem}

\begin{remark} The isomorphism in Theorem \ref{mainconj} (1) is canonical, so we identify the two crystals.
\end{remark}

\subsection{Specializing the quantum alcove model to type $A$}
\label{subsection:TypeA_alcove}

We start with the basic facts about the root system of type $A_{n-1}$. 
We can identify the space $\h_\R^*$ with the quotient $V:=\R^n/\R(1,\ldots,1)$,
where $\R(1,\ldots,1)$ denotes the subspace in $\R^n$ spanned 
by the vector $(1,\ldots,1)$.  
Let $\varepsilon_1,\ldots,\varepsilon_n\in V$ 
be the images of the coordinate vectors in $\R^n$.
The root system is 
$\Phi=\{\alpha_{ij}:=\varepsilon_i-\varepsilon_j \::\: i\ne j,\ 1\leq i,j\leq n\}$.
The simple roots are $\alpha_i=\alpha_{i,i+1}$, 
for $i=1,\ldots,n-1$. The highest root $\widetilde{\alpha}=\alpha_{1n}$. We let
$\alpha_0=\theta=\alpha_{n1}$.
%NOTE: talking about \alpha_0 here for future reference
The weight lattice is $\Lambda=\Z^n/\Z(1,\ldots,1)$. The fundamental weights are $\omega_i = \varepsilon_1+\ldots +\varepsilon_i$, 
for $i=1,\ldots,n-1$. 
A dominant weight $\lambda=\lambda_1\varepsilon_1+\ldots+\lambda_{n-1}\varepsilon_{n-1}$ is identified with the partition $(\lambda_{1}\geq \lambda_{2}\geq \ldots \geq \lambda_{n-1}\geq\lambda_n=0)$ having at most $n-1$ parts. 
Note that $\rho=(n-1)\varepsilon_1+(n-2)\varepsilon_2 + \ldots + 1\varepsilon_{n-1}$.
Considering the Young diagram of the dominant weight $\lambda$ as a concatenation of columns, 
whose heights are $\lambda_1',\lambda_2',\ldots$, 
corresponds to expressing $\lambda$ as $\omega_{\lambda_1'}+\omega_{\lambda_2'}+\ldots$ 
(as usual, $\lambda'$ is the conjugate partition to $\lambda$). 

The Weyl group $W$ is the symmetric group $S_n$, which acts on $V$ by permuting the coordinates $\varepsilon_1,\ldots,\varepsilon_n$. Permutations $w\in S_n$ are written in one-line notation $w=w(1)\ldots w(n)$. For simplicity, we use the same notation $(i,j)$ with $1\le i<j\le n$ for the root $\alpha_{ij}$ and the reflection $s_{\alpha_{ij}}$, which is the transposition $t_{ij}$ of $i$ and $j$.  

We now consider the specialization of the alcove model to type $A$. 
For any $k=1, \ldots , n-1$, we have the following $\omega_k$-chain, from $A_{\circ}$ to 
$A_{-\omega_k}$, denoted by 
	$\Gamma(k)$: %[\cite{lapawg}]:
	\begin{equation}
		\begin{matrix*}[l]
			( (k,k+1),&    (k,k+2)&,     \ldots,&  (k,n),    \\
			\phantom{(} (k-1,k+1),&  (k-1,k+2)&,     \ldots,& (k-1,n),  \\
			%		  &           \ldots  &                &       \nonumber \\
			\phantom{(,}\vdots & \phantom{,}\vdots & & \phantom{,}\vdots \\
			\phantom{(}(1,k+1),&   (1, k+2)&,      \ldots,&  (1,n)) \,.
		\end{matrix*}
		\label{eqn:lambdachainA_alcove}
	\end{equation}

	Fix a dominant weight $\lambda$, for which we use the partition notation above. We construct a $\lambda$-chain $\Gamma=\left( \beta_1, \beta_2, \dots , \beta_m \right)$ as the concatenation 
	$\Gamma:= \Gamma^{1}\dots\Gamma^{\lambda_1}$, where $\Gamma^{j}=\Gamma(\lambda'_j)$.
	Let  $J=\left\{ j_1 < \dots < j_s \right\}$ be a set of folding positions in 
	 $\Gamma$, not necessarily admissible, and let $T$ be the corresponding list of roots of $\Gamma$, also viewed as transpositions. 
	 The factorization of $\Gamma$ induces a 
	 factorization on $T$ as 
	 $T=T^{1}T^2 \dots T^{\lambda_1}$. 
	 We denote by $T^{1} \dots T^{j}$ the permutation obtained by composing 
	 the transpositions in $T^{1}, \dots , T^{j}$ from left to right. 
For $w\in W$, written $w=w_1w_2\dots w_n$, let $w[i,j]=w_i\dots w_j$. 
	  
	  \begin{definition}
		  \label{definition:fill}
		   Let $\pi_{j}=\pi_{j}(T):=T^1 \dots T^j$. We define the \emph{filling map},
	  which associates with each $J\subseteq[m]$ a filling of the Young diagram $\lambda$, by 
	  \begin{equation}
		  \label{eqn:filling_map}
		  \fill(J)=\fill(T):=C_{1}\dots C_{\lambda_1}\,,
		  \;\mbox{ where } C_{i}:=\pi_i[1,\lambda'_i].
	  \end{equation}
	    We define the \emph{sorted filling map} $\sfill(J)$  by sorting ascendingly
		  the columns of $\fill(J)$.
	  \end{definition}
	  
%Note: This is really an isomorphism.
%Note: Mention type C and other types.
\begin{theorem}{\rm \cite{Lenart,lalgam}}
	\label{theorem:bijection_type_A}	
	The map $\sfill$ is the unique affine crystal isomorphism between $\A(\Gamma)$ and the subgraph of $\B$ consisting of the dual Demazure arrows. 
	In other words, given $\sfill(J)=b$, there is a dual Demazure arrow $b\rightarrow f_p(b)$ if and only if $f_p(J)\ne \Bzero$, and we have  $f_p(b)=\sfill(f_p(J))$.
	The map $\sfill$ also preserves weights, and translates the height statistic into the Lascoux-Sch\"utzenberger {\em charge} statistic on fillings {\rm \cite{lassuc}}. 
\end{theorem}
There is a similar result in type $C$ \cite{Lenart}, \cite{lascec}, \cite[Section 4.2]{lalgam}.

\section{The main results}
\label{section:Yang-Baxter}

In this section we realize the combinatorial $R$-matrix in terms of the quantum alcove model, and show that this model is independent of the choice of a $\lambda$-chain. We start with a preview of the main result. 
Let $\mathbf{p}$ be the composition ${\mathbf p}=(p_1,\ldots,
p_k)$, and let ${\mathbf p}'=(p_1',\ldots, p_k')$ be a
permutation of ${\mathbf p}$. 
Let
\begin{align}
%\lambda:=\omega_{p_1}+\cdots +\omega_{p_k}=\omega_{p_1'}+\cdots+\omega_{p_k'}\,.
%\\
&B^{\otimes \mathbf{p}}:=B^{p_1,1}\otimes
\cdots \otimes B^{p_k,1}\,,\quad B^{\otimes
    \mathbf{p'}}:=B^{p_1',1}\otimes \cdots \otimes B^{p_k',1}\;\nonumber \\
&\Gamma := \Gamma(p_1)  \cdots \Gamma(p_k)\,,\quad \Gamma' :=
\Gamma(p'_1) \cdots\Gamma(p'_2) \,,\label{ggp}
\end{align}
where $\Gamma(i)$ is an $\omega_i$-chain; thus, $\Gamma$ and $\Gamma'$ are $\lambda$-chains, where $\lambda:=\omega_{p_1}+\cdots +\omega_{p_k}$.  We will show that $\A( \Gamma)$ and $\A( \Gamma')$ are models for the isomorphic affine crystals $B^{\otimes \mathbf{p}}$ and  
$B^{\otimes \mathbf{p'}}$. Thus, we want to realize the combinatorial $R$-matrix as an affine crystal isomorphism between $\A( \Gamma)$ and $\A( \Gamma')$.

\begin{example}
    \label{example:yb_ex1}
    We illustrate the combinatorial $R$-matrix in type $A_2$. Let
 $ \mathbf{p}= ( 1,2,2,1 )$, $ \mathbf{p}' = (1,2,1,2)  $, so $\lambda=(4,2,0)$. Then
\[B^{\otimes \mathbf{p}} = B^{1,1}\otimes B^{2,1}\otimes {B^{2,1}\otimes
        B^{1,1}}\;\simeq\;B^{1,1}\otimes B^{2,1}\otimes
    {B^{1,1}\otimes B^{2,1}}=B^{\otimes \mathbf{p}'}\,.\]
We first note that in type $ A $ the combinatorial $R$-matrix can be realized by  
%In type $A$, it is realized by 
Sch\"utzenberger's jeu de taquin (sliding algorithm) on the last two columns, see \cite{fulyt}.
For example:
\[\tableau{{3}}\otimes\tableau{{2}\\{3}}\otimes\tableau{{1}\\{2}}\otimes\tableau{{3}}=\tableau{{3}}\otimes\tableau{{2}\\{3}}\otimes\tableau{{1}&{ \color{black} 3}\\{\color{black} 2}&{}}\;\mapsto\;\tableau{{3}}\otimes\tableau{{2}\\{3}}\otimes\tableau{{\color{black} 1}&{}\\{2}&{3}}\;
%    \mapsto\;
\]
\[
\mapsto\;\tableau{{3}}\otimes\tableau{{2}\\{3}}\otimes\tableau{&{1}\\{2}&{3}}=\tableau{{3}}\otimes\tableau{{2}\\{3}}\otimes\tableau{{2}}\otimes\tableau{{1}\\{3}}
\,. 
\]
    %Recall the setup
    %from Sections \ref{subsection:TypeA_alcove} and
    %\ref{subsection:TypeA_qalcove}.
We now demonstrate how to realize the combinatorial $R$-matrix in the
quantum alcove model.  Let $\Gamma$ and $\Gamma'$ be the $\lambda$-chains corresponding to $ \mathbf{p}$ and $ \mathbf{p}'$:
\begin{gather} \begin{array}{ccccccccccccc}\!\!\Gamma=(\,&\!\!\!\!\!\underline{(1,2)},&\!\!\!\!\underline{(1,3)}&\!\!\!\!\:|\:&\!\!\!\!\underline{(2,3)},&\!\!\!\!(1,3)&\!\!\!\!\:|\:&\!\!\!\!\color{black}{(2,3)},&\!\!\!\!\underline{\color{black}{(1,3)}}&\!\!\!\!\:|\:&\!\!\!\!\underline{\color{black}{(1,2)}},&\!\!\!\!\underline{(1,3)}\,&\!\!\!\!)\,,\\[3mm]
\!\!\Gamma'=(\,&\!\!\!\!\!\underline{(1,2)},&\!\!\!\!\underline{(1,3)}&\!\!\!\!\:|\:&\!\!\!\!\underline{(2,3)},&\!\!\!\!(1,3)&\!\!\!\!\:|\:&\!\!\!\!\underline{\color{black}{(1,2)}},&\!\!\!\!{\color{black}{(1,3)}}&\!\!\!\!\:|\:&\!\!\!\!\underline{\color{black}{(2,3)}},&\!\!\!\!\underline{(1,3)}\,&\!\!\!\!)\,.
\end{array}
\label{eqn:yb_ex1}
\end{gather}
  In \eqref{eqn:yb_ex1} we also showed (via the underlined positions in $\Gamma$ and $\Gamma'$) the choice of two admissible subsets,  $J=\left\{ 1,2,3,6,7,8 \right\}$ in $\A( \Gamma )$ and $J'=\left\{
    1,2,3,5,7,8 \right\}$ in $\A( \Gamma' )$. The
details of our construction will be given in Section \ref{ybmoves} (in particular, Example \ref{example:yb_ex2} is a continuation of the present one); for now, note that $J$ will correspond to $J'$ via our  construction. 
Observe that
\[\sfill(J) = \tableau{{3}}\otimes\tableau{{2}\\{3}}\otimes\tableau{{1}\\{2}}\otimes\tableau{{3}}\;
\mapsto\;\tableau{{3}}\otimes\tableau{{2}\\{3}}\otimes\tableau{{2}}\otimes\tableau{{1}\\{3}}
= \sfill(J')
\,, \]
so we recover the construction above in terms of jeu de taquin.
\end{example}

We construct a bijection between $\A( \Gamma )$ and $\A( \Gamma')$ by
generalizing the construction in \cite{lenccg}, which gives the
bijection in the classical case, where admissible subsets 
correspond to saturated chains in Bruhat order.  The bijection in \cite{lenccg}
is constructed by applying a sequence of operations called {\em Yang-Baxter moves}.

\subsection{Quantum Yang-Baxter moves} 
\label{ybmoves}

This section contains our main constructions. We use freely the notation related to the quantum alcove model in Section \ref{subsection:qalcove}. 

We start by recalling that there are only two reflection orderings on the positive roots corresponding to a dihedral Weyl group of order $2q$, that is, a Weyl group of type
$A_1\times A_1$, $A_2$, $C_2$, or $G_2$ (with $q=2, \,3,\, 4,\, 6$,
respectively). Let $\overline{\Phi}$ be the corresponding root system
with simple roots $\alpha$, $\beta$. The two reflection orderings on $\overline{\Phi}^+$ are given by the sequence
\begin{equation}\label{eqn:reflorder}\beta_1:=\alpha,\;\;\beta_2:=s_\alpha(\beta),\;\;\beta_3:=s_\alpha
s_\beta(\alpha),\;\;\ldots,\;\;\beta_{q-1}:=s_\beta(\alpha),\;\;\beta_q:=\beta\,,\end{equation}
and its reverse.  

Fix a dominant weight $\lambda$. Let us consider an index set 
\begin{equation}
\label{indexset}
I:=\{\overline{1}<\ldots<\overline{t}<1<\ldots<q<\overline{t+1}<\ldots<\overline{n}\}\,.
\end{equation}
 Let $\Gamma=\{\beta_i\}_{i\in I}$ be a $\lambda$-chain, denote $r_i:=s_{\beta_i}$ as before, and let $\Gamma'=\{\beta_i'\}_{i\in I}$ be the sequence of roots defined by 
\begin{equation}
\label{yblambdachain}
\beta_i'=\casetwox{\beta_{q+1-i}}{i\in [q]}{\beta_i}{i\in I\setminus[q]}
\end{equation}
In other words, the sequence $\Gamma'$ is obtained from the
$\lambda$-chain $\Gamma$ by reversing a certain segment. Now assume
that $\{\beta_1,\ldots,\beta_q\}$ are the positive roots (without repetition) of a rank two
root subsystem $\overline{\Phi}$ of $\Phi$. 
The corresponding dihedral reflection group is a subgroup of the Weyl group $W$. 

\begin{proposition}{\rm \cite{lenccg}}
    \label{reversing} 
    ~
    \begin{enumerate}
        \item[\rm{(1)}] The sequence $\Gamma'$ is also a $\lambda$-chain, and the sequence $(\beta_1,\ldots,\beta_q)$ is a {reflection ordering} of $\overline{\Phi}^+$.

        \item[\rm{(2)}] We can obtain any $\lambda$-chain from any other $\lambda$-chain by moves of the form $\Gamma\rightarrow\Gamma'$. 
    \end{enumerate}
\end{proposition}

Let us now map the admissible subsets in $\A(\Gamma)$ to
those in $\A(\Gamma')$. To this end, fix a reflection ordering of $\Phi^+$ compatible with the above ordering $(\beta_1,\ldots,\beta_q)$ of $\overline{\Phi}$; this clearly exists (take any reflection ordering of $\Phi^+$, and reverse it if needed). Now fix an admissible subset $J = \{j_1 < \dots < j_s \}$ in $\A(\Gamma)$. Define $w(J) := r_{j_1}r_{j_2}\cdots r_{j_s}$ and, by abusing notation, let
\begin{equation}\label{setup1}
u:=w(J\cap\{\overline{1},\ldots,\overline{t}\})\,,\;\;\;\;\mbox{and}\;\;\;\;w:=w(J\cap(\{\overline{1},\ldots,\overline{t}\}\cup[q]))\,.\end{equation}
% Also let 
%\begin{equation}
%\label{setup2}
%u=\floor{u}\overline{u}\,,\;\;\;\;w=\floor{w}\overline{w}\,,\;\;\;\;a:=\ell(\overline{u})\,,\;\;\;\;\mbox{and}\;\;\;\;b:=\ell(\overline{w})\,,
%\end{equation}
% as above. 

Note that, by the definition of $\A(\Gamma)$, we have a path in $\QB(W)$ from $u$ to $w$ with increasing edge labels $J\cap[q]$ (here we identify an edge label $\beta_i$ with $i$, for $i\in[q]$). By the shellability property of $\QB(W)$, that is, by Theorem \ref{thm:shell} (1), there is another path in $\QB(W)$ from $u$ to $w$ whose edge labels (in $\Phi^+$) increase with respect to the reverse of the reflection ordering considered above. In fact, by the proof of Theorem \ref{thm:shell} in \cite{bfpmbo}, these edge labels are also in $\overline{\Phi}^+$, since the edge labels of the first path had this property. (More precisely, we refer to the proof of \cite[Lemma 6.7]{bfpmbo}, which proves Theorem \ref{thm:shell} (2); here, the passage from an arbitrary path in $\QB(W)$ to a lexicographically minimal/maximal one is realized by successively changing length 2 subpaths, while staying in the same dihedral reflection group.) Thus, by now identifying the label $\beta_i'$ with $i$, we can view the edge labels of the new path as a subset of $[q]$, which we denote by $Y_{u,w}(J\cap[q])$. 
 
 It is clear that we have a bijection $Y\::\:\A(\Gamma)\rightarrow \A(\Gamma')$ given by
\begin{equation}
\label{setup3}
Y(J):=(J \backslash [q])\,\cup\, Y_{u,w}(J\cap[q]) \,.
\end{equation}
We call the moves $J\mapsto Y(J)$ {\em quantum Yang-Baxter moves}. They generalize the Yang-Baxter moves in  
\citep{lenccg}, which correspond to saturated chains in the Bruhat order (i.e., there are no down steps). The explicit description of the quantum Yang-Baxter moves is given in Section \ref{explicitmoves}.

\begin{example}\label{example:yb_ex2}
    We continue Example \ref{example:yb_ex1}. In
    \eqref{eqn:yb_ex1} $ \Gamma  $ is obtained from $ \Gamma'  $ by
    reversing the first three roots in the right half of $ \Gamma' $.
    We will use the conventions of this section and (re)label this
    segment of $ \Gamma' $ by $ (\beta_1, \beta_2, \beta_3) $.
    We have ${u} = s_1s_2s_1=321$, and ${w}= s_1=213$, 
     in one line notation.
    In this case, we have 
    $ Y_{u,w}(\left\{ 1,3 \right\})
    =
    \left\{ 2,3 \right\} 
    $. 
    See Figure \ref{fig:A2qbg}; note that this figure is with respect to the
    reflection ordering $(\beta_1, \beta_2, \beta_3)$. 
    Hence 
    $Y
    (\left\{ 1,2,3,5,7,8 \right\})
    =
    \left\{ 1,2,3,6,7,8 \right\}
    $, where now the indexing corresponds to the entire $\lambda$-chain, cf. Example \ref{example:yb_ex1}.
    %NOTE: this example really uses the other reflection order, which
    %happens to be the same.
\end{example}

\subsection{Properties of the quantum Yang-Baxter moves}\label{prop-ybmoves}

In this section we study the main properties of the quantum Yang-Baxter moves, which are concerned with various quantities they preserve, as well as with their interaction with the crystal operators.
We use the same notation as in Section \ref{ybmoves}.

We get started by noting that the quantum Yang-Baxter moves preserve the Weyl group element
$w(\,\cdot\,)$ associated to an admissible subset, that is, 
\begin{equation}
\label{w-pres}
w(Y(J))=w(J)\,.
\end{equation}
Indeed, a quantum Yang-Baxter move simply replaces a subpath of the path \eqref{eqn:admissible} in $\QB(W)$ corresponding to $J$ with another subpath, between the same elements (denoted $u$ and $w$ in \eqref{setup1}). 

In order to prove the next two properties, we need two lemmas. Let $H_{\beta_i,-l_i}$, for $i=1,\ldots,q$, be the hyperplanes corresponding to the roots $\beta_1,\ldots,\beta_q$ in $\Gamma$, cf. the notation in Section \ref{subsection:qalcove}.

\begin{lemma}\label{height-formula}
Let $\beta_i^\vee=a\beta_1^\vee+b\beta_q^\vee$ with $a,b\in\Z_{\ge 0}$. Then $l_i=al_1+bl_q$. 
\end{lemma}

\begin{proof} We proceed by induction on the height of $\beta_i^\vee$ in $\overline{\Phi}$, i.e., on $a+b$. If $a+b=1$, the statement is obvious, so we assume $a+b>1$; in fact, this means that $a,b\ge 1$. We must have $\inner{\beta_i}{\beta_1^\vee}>0$ or $\inner{\beta_i}{\beta_q^\vee}>0$, because otherwise the contradiction $2=\inner{\beta_i}{\beta_i^\vee}\le 0$ would follow. Without loss of generality, assume that $c:=\inner{\beta_1}{\beta_i^\vee}>0$. Then we have 
\[s_{\beta_1}(\beta_i^\vee)=\beta_i^\vee-c\beta_1^\vee=(a-c)\beta_1^\vee+b\beta_q^\vee\,.\]
Since $b\ge 1$, we must have $a-c\ge 0$. So $s_{\beta_1}(\beta_i^\vee)$ is some positive coroot $\beta_j^\vee$ in $\overline{\Phi}^\vee$ of smaller height than $\beta_i^\vee$. By induction, we have
\begin{equation}\label{fact1}l_j=(a-c)l_1+bl_q\,.\end{equation} On another hand, by using the fact that $\beta_i^\vee=\beta_j^\vee+c\beta_1^\vee$, and by applying \cite{lapcmc}[Propositions 4.4 and 10.2], we deduce 
\begin{equation}\label{fact2}l_i=l_j+cl_1\,.\end{equation}
 The induction step is completed by combining \eqref{fact1} and \eqref{fact2}.
\end{proof}

We easily derive another lemma.

\begin{lemma}\label{nonempty-int}
The intersection of affine hyperplanes $\bigcap_{i=1}^q H_{\beta_i,-l_i}$ has codimension $2$, and coincides with $H_{\beta_1,-l_1}\cap H_{\beta_q,-l_q}$.
\end{lemma}

\begin{proof}
Consider an element $\mu$ in $H_{\beta_1,-l_1}\cap H_{\beta_q,-l_q}$, which has codimension 2. By Lemma \ref{height-formula}, and using the related notation, we have
\[\inner{\mu}{\beta_i^\vee}=a\inner{\mu}{\beta_1^\vee}+b\inner{\mu}{\beta_q^\vee}=-al_1-bl_q=-l_i\,,\]
so $\mu$ also lies on the hyperplane $H_{\beta_i,-l_i}$. 
\end{proof}

We can now prove the following property of the quantum Yang-Baxter moves. Let $\widehat{r}_i$, for $i=1,\ldots,q$, be the affine reflections in the hyperplanes $H_{\beta_i,-l_i}$, cf. the notation in Section \ref{subsection:qalcove}. Let $\widehat{r}_i':=\widehat{r}_{q+1-i}$ and $r_i':=s_{\beta_i'}=s_{\beta_{q+1-i}}$, cf. \eqref{yblambdachain}. 

\begin{proposition}\label{prop:yb-weight}
        The quantum Yang-Baxter moves preserve the associated weight (defined in {\rm \eqref{def:weight}}), i.e. 
    \begin{equation}
        \mu(J) = \mu(Y(J))\,,
        \label{eqn:yb-weight}
    \end{equation}
      where the left hand side is computed with respect to $\Gamma$, and
    the right hand side is computed with respect to $\Gamma'$.
\end{proposition}

\begin{proof} Let $J\cap [q]=\{j_1<\ldots<j_s\}$ and $Y(J)\cap [q]=\{j_1'<\ldots<j_s'\}$. Then, by picking $\mu$ in $\bigcap_{i=1}^q H_{\beta_i,-l_i}$, cf. Lemma \ref{nonempty-int}, we have
\begin{align*}
\widehat{r}_{j_1}\ldots \widehat{r}_{j_s}(\nu)&=\widehat{r}_{j_1}\ldots \widehat{r}_{j_s}(\mu+(\nu-\mu))\\&=\mu+r_{j_1}\ldots r_{j_s}(\nu-\mu)\\
&=\mu+r_{j_1'}\ldots r_{j_s'}(\nu-\mu)\\&=\widehat{r}_{j_1}'\ldots \widehat{r}_{j_s}'(\nu)\,.
\end{align*}
Here the second and last equalities follow from the fact that $\widehat{r}_i(x+y)=\widehat{r}_i(x)+r(y)$ and the choice of $\mu$, while the third one is based on the definition of a quantum Yang-Baxter move. Now \eqref{eqn:yb-weight} follows directly from the definition of $\mu(J)$.
\end{proof}

Another property of the quantum Yang-Baxter moves is the following.

\begin{proposition}\label{pres-energy}
    The quantum Yang-Baxter moves preserve the height statistic (defined in {\rm \eqref{eqn:height_statistic}}), i.e.
    \begin{equation}
        \height(J)=\height(Y(J))\,,
        \label{eqn:yb-height}
    \end{equation}
    where the left hand side is computed with respect to $\Gamma$, and
    the right hand side is computed with respect to $\Gamma'$.
 \end{proposition}

\begin{proof}
The classical Yang-Baxter moves involve no down steps, so the corresponding height statistic is $0$, and \eqref{eqn:yb-height} is immediate. Now assume that the down steps in $J\cap[q]$ are in positions $k_1<\ldots<k_r$, and those in  $Y(J)\cap[q]$ in positions $k_1'<\ldots<k_t'$ (note that $r$ is not necessarily equal to $t$). By \citep[Lemma 1 (2)]{posqbg}, we have
\begin{equation}\label{eqdown}\beta_{k_1}^\vee+\ldots+\beta_{k_r}^\vee=\beta_{q+1-k_1'}^\vee+\ldots+\beta_{q+1-k_t'}^\vee\,;\end{equation}
here we need the fact that the paths between $u$ and $w$ in $\QB(W)$ corresponding to $J$ and $Y(J)$ are shortest ones, by Theorem \ref{thm:shell} (2). In fact, by examining the explicit form of the quantum Yang-Baxter moves in Section \ref{explicitmoves}, we can see that $r,t\le 2$, although this is not needed here. 

Using the fact that $J\setminus[q]=Y(J)\setminus[q]$, and also the equality between the inner products of the two sides of \eqref{eqdown} with $\lambda$, it suffices to prove that
\begin{equation}\label{eqheight} l_{k_1}+\ldots+l_{k_r}=l_{q+1-k_1'}+\ldots+l_{q+1-k_t'}\,,\end{equation}
cf. the definition of $\height(J)$. By expressing each root $\beta_i^\vee$ in \eqref{eqdown} as $a_i\beta_1^\vee+b_i\beta_q^\vee$ (with $a_i,b_i\in\Z_{\ge 0}$), and by using \eqref{eqdown}, we derive
\[a_{k_1}+\ldots+a_{k_r}=a_{q+1-k_1'}+\ldots+a_{q+1-k_t'}\,,\;\;\;\;\;b_{k_1}+\ldots+b_{k_r}=b_{q+1-k_1'}+\ldots+b_{q+1-k_t'}\,.\]
The formula \eqref{eqheight} now follows by using Lemma \ref{height-formula} to express $l_i=a_il_1+b_il_q$. 
\end{proof}

The following theorem will be proved Section \ref{someproofs}. Note that Proposition 
\ref{prop:yb-weight}  and Theorem \ref{theorem:comm-f-y} generalize the similar results in \citep{lenccg} for 
the classical Yang-Baxter moves.

\begin{theorem}\label{theorem:comm-f-y} 
    The crystal operators commute with 
    the quantum Yang-Baxter moves, that is, $f_p$ is defined on an admissible subset $J$
    if and only if it is defined on $Y(J)$, and we have
\[
    Y(f_p(J))=f_p(Y(J))\,.
\]
\end{theorem}

\subsection{Corollaries and conjectures}

In this section we state some corollaries of the results in the previous section which are, in fact, the main results of the paper. We also discuss possible strengthenings of these results.

%Theorem \ref{mainconj}, 
Theorem \ref{theorem:comm-f-y}, Proposition \ref{prop:yb-weight}, and Proposition \ref{pres-energy} immediately imply the following corollary (cf. also Proposition \ref{reversing}), which essentially says that the quantum alcove model is independent of the choice of a $\lambda$-chain.

\begin{corollary}\label{cor0}
Given two $\lambda$-chains $\Gamma$ and $\Gamma'$, there is an affine crystal isomorphism between $\A(\Gamma)$ and $\A(\Gamma')$ which preserves the weights and heights of the vertices. This isomorphism is realized as a composition of quantum Yang-Baxter moves.
\end{corollary}

By composing the explicit bijection between a tensor product of (column shape) KR crystals $B^{\otimes \mathbf{p}}$ and $\A(\Gamma_{\rm lex})$ in Theorem \ref{mainconj} with an affine crystal isomorphism between $\A(\Gamma_{\rm lex})$ and $\A(\Gamma)$ realized by quantum Yang-Baxter moves (where $\Gamma$ is an arbitrary $\lambda$-chain, see Corollary \ref{cor0}), we obtain the following strengthening of Theorem \ref{mainconj}.

\begin{corollary}\label{cor1}
    Theorem {\rm \ref{mainconj}} holds for any choice of a $\lambda$-chain (instead of a lex $\lambda$-chain), based on the bijection mentioned above. 
   \end{corollary}

\begin{remark}\label{affcris0}
There are several ways to connect two $\lambda$-chains $\Gamma$ to $\Gamma'$ via the moves in Proposition \ref{reversing}. A priori, the corresponding compositions of quantum Yang-Baxter moves give different affine crystal isomorphisms between $\A(\Gamma)$ and $\A(\Gamma')$ in Corollary \ref{cor0}. Therefore, in Corollary \ref{cor1} we have a collection of a priori different affine crystal isomorphisms between   $B^{\otimes \mathbf{p}}$ and $\A(\Gamma)$, for a fixed $\lambda$-chain $\Gamma$. All this is due to the fact that $B^{\otimes \mathbf{p}}$ is not necessarily connected under the dual Demazure arrows, cf. Remark \ref{rem:conn} (1) and Remark \ref{rem:demazure} (2). However, we make the following conjecture.
\end{remark}

\begin{conjecture}\label{conj1}
All the affine crystal isomorphisms between $\A(\Gamma)$ and $\A(\Gamma')$ in Corollary {\rm \ref{cor0}} are identical. The same is true about the isomorphisms between $B^{\otimes \mathbf{p}}$ and $\A(\Gamma)$ in Corollary {\rm \ref{cor1}}.
\end{conjecture}

In order to prove this conjecture, we need $\A(\Gamma)$ to be connected (see Remark \ref{affcris0}). So we need to realize the non-dual Demazure $0$-arrows in the quantum alcove model (in addition to the dual Demazure arrows), and prove the corresponding strengthening of Theorem \ref{theorem:comm-f-y}. Note that the non-dual Demazure $0$-arrows are realized in \cite{unialcmod2} in an analogous model, namely the {\em quantum LS-path model}. Thus, the challenge is to translate this construction into the setup of the quantum alcove model, via the bijection in \cite{unialcmod2} between the two mentioned models. The new construction will be  considerably more involved than the one in \eqref{eqn:rootF} for the dual Demazure arrows, cf. \cite{lalgam}[Example 4.9].

\begin{remark}\label{affcris} If all the tensor factors in $B^{\otimes \mathbf{p}}$ are perfect crystals, then there is a unique affine crystal isomorphism between $B^{\otimes \mathbf{p}}$ and $\A(\Gamma)$, by Remark \ref{rem:demazure} (2). Conjecture \ref{conj1} then follows. 
\end{remark}

Now let $\mathbf{p}$ be a composition and $ \mathbf{p'}$  a permutation of it. Recall the corresponding $\lambda$-chains $\Gamma$ and $\Gamma'$ constructed in \eqref{ggp}. By Corollary \ref{cor1}, we have affine crystal isomorphisms between $B^{\otimes \mathbf{p}}$ and $\A(\Gamma)$, as well as  between $B^{\otimes \mathbf{p'}}$ and $\A(\Gamma')$. Also recall from Remark \ref{rem:conn} (2) that there is a unique affine crystal isomorphism $ B^{\otimes \mathbf{p} }~ \cong~B^{\otimes
       \mathbf{p'} } $, namely the combinatorial $R$-matrix. The uniqueness property in Remark \ref{affcris} leads to the following result.

%In future work we will investigate the previous corollary for non-perfect crystals.
\begin{corollary}\label{cor2}
    Suppose that all the tensor factors in $B^{\otimes \mathbf{p} }$ are perfect crystals.
   Then the quantum Yang-Baxter moves realize the combinatorial $R$-matrix as an affine crystal isomorphism between $\A(\Gamma)$ and $\A(\Gamma')$, in the sense mentioned above.
    \label{corollary:R-matrix}
\end{corollary}

\begin{remarks} (1) If the strengthening of  Theorem \ref{theorem:comm-f-y} corresponding to the non-dual Demazure $0$-arrows (see above) is proved, then Corollary \ref{cor2} holds in full generality, not just in the perfect case.  

 (2) In Examples \ref{example:yb_ex1} and \ref{example:yb_ex2}, we showed how we can specialize our type-independent construction of the combinatorial $R$-matrix to the tableau model in type $A$; the specialization amounts to permuting strictly increasing columns, which can also be done using Sch\"utzenberger's {jeu de taquin} (sliding algorithm) on two columns \cite{fulyt}. The crucial ingredient for the specialization is the affine crystal isomorphism from the quantum alcove model to the tableau model which was described in Section \ref{subsection:TypeA_alcove}. A similar isomorphism to the corresponding tableau model, based on Kashiwara-Nakashima (KN) columns \cite{kancgr}, was constructed in type $C$ in \cite{Lenart,lalgam}, and one is being developed in type $B$ in \cite{BL}. These maps can also be used to specialize our construction to the corresponding tableau models, i.e., to permute KN columns. Note that the corresponding extensions of the sliding algorithm are involved and not transparent, see \cite{lecsc,lecst}.   
\end{remarks}

\section{Proof of Theorem \ref{theorem:comm-f-y}}\label{someproofs}

In Section \ref{pf1} we present a geometric interpretation of the quantum alcove model, as well as some additional related results, all of which are needed in the proof of Theorem \ref{theorem:comm-f-y}. In Section \ref{pf2} we derive several lemmas, after which we complete the proof in Section \ref{pf3}.

\subsection{Revisiting the quantum alcove model}\label{pf1}

In this section we use freely the notation from Section \ref{subsection:qalcove} (as opposed to the one in Section \ref{ybmoves}). We fix an admissible subset $J = \left\{ j_1  < \ldots < j_s \right\} $ in ${\mathcal A}(\Gamma)$, and recall the corresponding roots $\gamma_i$ defined in \eqref{defgamma}, which form the sequence $\Gamma(J)$. Fix also an index $p$ for a crystal operator, between $0$ and the rank of the root system $\Phi$. We will need the following lemma from \cite{lalgam}.

\begin{lemma}{\rm \cite{lalgam}}
		\label{lemma:admissible2}
		 Assume that $r_{j_a}\dots r_{j_1}(\alpha_p)>0$ and 
		$r_{j_{b}}\dots r_{j_1}(\alpha_p)<0$, for $0\leq a < b$
		(if $a=0$, then the first condition is void). Then there exists $i$ with $a < i \leq b$ such that 
		$\gamma_{j_{i}}=\alpha_p$.
	\end{lemma}

We will present a simpler encoding of (the graph of) the piecewise-linear function $g_{\alpha_p}$, on which the construction of $f_p(J)$ is based. Recall that the positions of the roots $\pm\alpha_p$ in the sequence $\Gamma(J)$ are recorded in the sequence $I_{\alpha_p}(J)$, see \eqref{eqn:ILalpha}. We associate with each pair 
$( \gamma_i,\, \varepsilon_i)$ for $i \in I_{ \alpha_p}(J)$  
a letter in the alphabet $A = \{ +, - , \pm , \mp\}$ as follows:  
\begin{equation}
    \label{eqn:gword}
 (\alpha_p,1) \mapsto +,\,\, (-\alpha_p, 1) \mapsto -,\,\,
  (\alpha_p,-1) \mapsto \pm,\,\, (-\alpha_p, -1) \mapsto \mp. 
\end{equation}
We also represent $\sgn( \inner{\gamma_{ \infty }}{ \alpha_p^{\vee}})$ by
its sign. 
In this way, we can encode $g_{ \alpha_p}$ by a word in the alphabet $A$, which will be denoted by $S_p$.

\begin{example}
    \label{example:gword}
   The above encoding can be applied to any root $\alpha$. For instance, we can encode $g_{ \alpha }$ in Figure \ref{fig:galpha} by the word
   $\pm-++\pm-\pm++$.
\end{example}

\begin{lemma}{\rm \cite{lalgam}}
    \label{proposition:pm_condition}
    In the word $S_p$ associated to $g_{ \alpha_p}$, a non-terminal
    symbol from the set $\{ +, \mp \}$ must be followed by a symbol
    from the set $\{ + , \pm \}$.
\end{lemma}

We will now present a geometric realization of the admissible subset $J$, and discuss the related interpretation of the word $S_p$. 
The relevant structure, called a {\em gallery} (cf. \cite{gallsg}), is a sequence 
\begin{equation}\label{pij}\Pi=(F_0,A_0,F_1,A_1, F_2, \dots , F_m, A_m, F_{m+1})\end{equation}
such that $A_0,\dots,A_m$ are alcoves;
$F_j$ is a codimension one common face of the alcoves $A_{j-1}$ and $A_j$,
for $j=1,\dots,m$; $F_0$ is a vertex of the first alcove $A_0$; and 
$F_{m+1}$ is a vertex of the last alcove $A_m$. 
Furthermore, we require that $F_0=\{0\}$, $A_0=A_\circ$, and 
$F_{m+1}=\{\mu\}$ for some weight $\mu$,
which is called the weight of the gallery. 
The folding operator $\phi_j$ is the operator which acts on a gallery by  leaving its initial segment from $A_0$ to $A_{j-1}$ intact, and by reflecting the remaining tail in the affine hyperplane containing the face $F_j$. In other words, we define
$$\phi_j(\Pi):=(F_0,A_0, F_1, A_1, \dots, A_{j-1}, F_j'=F_j,  A_{j}', F_{j+1}', A_{j+1}', \dots,  A_m', F_{m+1}'),$$
where $F_j\subset H_{\alpha,k}$, $A_i' := s_{\alpha,k}(A_i)$, and $F_i':=s_{\alpha,k}(F_i)$, for $i=j,\dots,m+1$.
 %We say that a gallery is {\it unfolded\/} if $A_{j-1}\ne A_j$, for $j=1,\dots,l$.  

 Our fixed alcove path \eqref{alcovepath} determines an obvious gallery $\Pi(\emptyset)$ 
%\[\Pi(\emptyset)=(F_0,A_0,F_1,\dots,F_m, A_m, F_{m+1})\]
 of weight $-\lambda$. Given the fixed admissible subset $J$, we define the gallery $\Pi(J)$ as $\phi_{j_1}\cdots \phi_{j_s} (\Pi(\emptyset))$. We denote $\Pi(J)$ as in \eqref{pij}. 
We can easily detect the folding positions in $J$ based on the gallery $\Pi(J)$: given $i\in [m]$, we have $i\in J$ (so $\varepsilon_i=-1$) if and only if $A_{i-1}=A_i$. By definition, the root $\gamma_i$ in $\Gamma(J)$ is orthogonal to the hyperplane containing the face $F_i$. Moreover, if $i\not\in J$, then $\gamma_i$ points in the direction from $A_i$ to $A_{i-1}$; otherwise, it points in the direction from the face $F_i$ towards $A_{i-1}=A_i$. 

We conclude this section by interpreting the signs in $S_p$. We need to consider the positive side of an {\em oriented} affine hyperplane $\stackrel{\rightarrow}{H}_{\alpha,k}$, which is the same as the negative side of $\stackrel{\rightarrow}{H}_{-\alpha,-k}$, and which consists of all $\mu$ with $\inner{\mu}{\alpha^\vee}>k$. Fix $i$ in $I_{\alpha_p}(J)$, so $(\gamma_i,\varepsilon_i)=(\pm\alpha_p,\pm 1)$. Based on the above discussion, we have the following cases related to the position of the alcoves $A_{i-1}$ and $A_i$ in $\Pi(J)$ with respect to the oriented hyperplane $\stackrel{\rightarrow}{H}_{\alpha_p,-\sgn(\alpha_p)l_i^J}$, cf. \eqref{deflev} and \eqref{eqn:gword}:
\begin{itemize}
\item if $(\gamma_i,\varepsilon_i)=(\alpha_p,1) \mapsto +$, then $A_{i-1}$ is on the positive side, and $A_i$ on the negative one;
\item if $(\gamma_i,\varepsilon_i)=(-\alpha_p,1) \mapsto -$, then $A_{i-1}$ is on the negative side, and $A_i$ on the positive one;
\item if $(\gamma_i,\varepsilon_i)=(\alpha_p,-1) \mapsto \pm$, then $A_{i-1}=A_i$ are on the positive side;
\item if $(\gamma_i,\varepsilon_i)=(-\alpha_p,-1) \mapsto \mp$, then $A_{i-1}=A_i$ are on the negative side.
\end{itemize}

\subsection{Lemmas}\label{pf2} We continue using the above setup, except that now the indexing set $I$ for the $\lambda$-chain is the one defined in \eqref{indexset}; we also use the rest of the notation in Section \ref{ybmoves}. 

We start by observing that, by Lemma \ref{nonempty-int}, the reflections in the affine hyperplanes 
\begin{equation}\label{rank2}
H_{\beta_i,-l_i}\,,\;\;\;\;\;\mbox{for $i\in[q]$}\,.
\end{equation}
 are the reflections of a rank 2 root system. We will use this fact implicitly below.

As above, $J$ is a fixed admissible subset in ${\mathcal A}(\Gamma)$, and $p$ is an index for a crystal operator. Denoting, as usual, $\widehat{r}_i:=s_{\beta_i,-l_i}$ for any $i\in I$, we let $\widehat{w}(J):=\widehat{r}_{j_1}\ldots\widehat{r}_{j_s}$; then, by abusing notation, we set
\begin{equation}\label{setup2}
\widehat{u}:=\widehat{w}(J\cap\{\overline{1},\ldots,\overline{t}\})\,,\;\;\;\;\mbox{and}\;\;\;\;\widehat{w}:=\widehat{w}(J\cap(\{\overline{1},\ldots,\overline{t}\}\cup[q]))\,,\end{equation}
cf. \eqref{setup1}. Now consider the images of the affine hyperplanes in \eqref{rank2} under $\widehat{u}$, namely
\begin{equation}\label{affhypu}
\widehat{u}(H_{\beta_i,-l_i})=H_{u(\beta_i),m_i}\,,\;\;\;\;\;\mbox{for $i\in[q]$}\,.
\end{equation}
The roots $\{\gamma_1,\ldots,\gamma_q\}$ form a subset of $\{\pm u(\beta_1),\ldots,\pm u(\beta_q)\}$, because $\{\pm\beta_1,\ldots,\pm\beta_q\}$ is invariant under the corresponding reflections. 

Assuming that $u(\beta_k)=\pm\alpha_p$ for some $k\in[q]$, which is unique, we will refer to the following hyperplane in \eqref{affhypu}, with the appropriate orientation:
\begin{equation}\label{hypp}\stackrel{\rightarrow}{H}:=\stackrel{\rightarrow}{H}_{\alpha_p,\sgn(u^{-1}(\alpha_p))m_k}\,.\end{equation}
Note that, by Lemma \ref{nonempty-int}, if the set $I_{\alpha_p}(J)\cap [q]$ is non-empty, then for any $i$ in this set we have $\stackrel{\rightarrow}{H}=\stackrel{\rightarrow}{H}_{\alpha_p,-\sgn(\alpha_p)l_i^J}$; in particular, the heights $l_i^J$ corresponding to the root $\alpha_p$ are identical in the ``$[q]$-window'' given by the indices $1,\ldots,q$ in the set $I$. 

We continue to use the notation \eqref{pij} for the gallery $\Pi(J)$, except that now the indexing set $I$ is the one in \eqref{indexset}; more precisely, we write
\[\Pi(J)=(\{0\},A_\circ,F_{\overline{1}},A_{\overline{1}},F_{\overline{2}},\ldots,F_{\overline{t}},A_{\overline{t}},F_1,A_1,F_2,\ldots,F_q,A_q,F_{\overline{t+1}},\ldots,F_{\overline{n}},A_{\overline{n}},\{-\mu(J)\})\,.\] 
We use a similar notation for $\Pi(\emptyset)$, namely we denote its alcoves and their corresponding faces by $A_i^\emptyset$ and $F_i^\emptyset$, for $i\in I$. 

\begin{lemma}\label{lemuw} Assuming that $u(\beta_k)=\pm\alpha_p$ for some $k\in[q]$, we have:
\begin{enumerate}
\item[{\rm (1)}] $u^{-1}(\alpha_p)>0$ if and only if $A_{\overline{t}}$ is on the positive side of $\stackrel{\rightarrow}{H}$;
\item[{\rm (2)}] $w^{-1}(\alpha_p)<0$ if and only if $A_{q}$ is on the positive side of $\stackrel{\rightarrow}{H}$.
\end{enumerate}
\end{lemma}

\begin{proof}
Note first that $A_{\overline{t}}^\emptyset$ is on the positive side of all hyperplanes in \eqref{rank2}, more precisely of $\stackrel{\rightarrow}{H}_{\beta_i,-l_i}$, whereas $A_q^\emptyset$ is on the negative side. Clearly, $A_{\overline{t}}^\emptyset$ is on the same side with respect to $\stackrel{\rightarrow}{H}_{\beta_k,-l_k}$ as $A_{\overline{t}}=\widehat{u}(A_{\overline{t}}^\emptyset)$ is with respect to $\stackrel{\rightarrow}{H}_{u(\beta_k),m_k}$. So $A_{\overline{t}}$ is on the side of $\stackrel{\rightarrow}{H}$ of sign $\sgn(u^{-1}(\alpha_p))$. 

For the second part, let $v:=w(J\cap[q])$ and $\widehat{v}:=\widehat{w}(J\cap[q])$, cf. \eqref{setup1} and \eqref{setup2}, so we have $w=uv$ and $\widehat{w}=\widehat{u}\widehat{v}$. First note that, since $A_q^\emptyset$ is on the negative side of $\stackrel{\rightarrow}{H}_{\beta_k,-l_k}$, the alcove $\widehat{v}(A_q^\emptyset)$ is on the side of the same hyperplane of sign $-\sgn(v^{-1}(\beta_k))$; here we used the straightforward fact that, if $\mu$ is on the positive side of $\stackrel{\rightarrow}{H}_{\alpha,0}$ and $x$ is a Weyl group element, then $x(\mu)$ is on the side of the same hyperplane of sign $\sgn(x^{-1}(\alpha))$. Then, by the same reasoning as in the proof of the first part, we deduce that $A_q=\widehat{w}(A_q^\emptyset)=\widehat{u}(\widehat{v}(A_q^\emptyset))$ is on the side of $\stackrel{\rightarrow}{H}$ of sign
\[-\sgn(v^{-1}(\beta_k))\cdot\sgn(u^{-1}(\alpha_p))=-\sgn(v^{-1}(u^{-1}(\alpha_p)))=-\sgn(w^{-1}(\alpha_p))\,;\]
here the first equality follows from the hypothesis, since $u^{-1}(\alpha_p)=\sgn(u^{-1}(\alpha_p))\beta_k$. This concludes the proof.
\end{proof}

Now recall from Section \ref{pf1} the word $S_p$ in the alphabet $\{ +, - , \pm , \mp\}$ encoding the function $g_{\alpha_p}$, and from Section \ref{subsection:qalcove} the way in which this function is used to construct $f_p(J)$. We will see that the subword $\overline{S}_p$ of $S_p$ which corresponds to the $[q]$-window can have only the following forms:
\begin{enumerate}
        \item[(a)] $\phantom{-}\pm \ldots \pm$, with at least one symbol $\pm$
        \item[(b)] $\phantom{-}\pm \ldots \pm + $ 
        \item[(c)] $-\pm \ldots \pm$ 
           \item[(d)] $-\pm \ldots \pm + $
        \item[(e)] $\phantom{-}\:\mp$
            \item[(f)] empty word,
        \end{enumerate}
    where in cases (b)$-$(d) 
   % \eqref{enumerate:window_graph_x}-\eqref{enumerate:window_graph_mxp}
    the string $ \pm \ldots \pm  $ can be empty. 

Using Lemma \ref{lemuw}, we now give easy criteria for these cases based on the Weyl group elements $u$ and $w$ corresponding to $J$.

\begin{lemma}\label{critcase} The following hold.
\begin{enumerate} 
\item[\rm{(1)}] If $\alpha_p\not\in\{\pm u(\beta_1),\ldots,\pm u(\beta_q)\}$, then we are in case {\rm (f)}.
\item[\rm{(2)}] If $\alpha_p\in\{\pm u(\beta_1),\ldots,\pm u(\beta_q)\}$, $u^{-1}(\alpha_p)<0$, and $w^{-1}(\alpha_p)<0$, then we are in case {\rm (c)}.
\item[\rm{(3)}] If $\alpha_p\in\{\pm u(\beta_1),\ldots,\pm u(\beta_q)\}$, $u^{-1}(\alpha_p)<0$, and $w^{-1}(\alpha_p)>0$, then we are in cases {\rm (d)$-$(f)}.
\item[\rm{(4)}] If $\alpha_p\in\{\pm u(\beta_1),\ldots,\pm u(\beta_q)\}$, $u^{-1}(\alpha_p)>0$, and $w^{-1}(\alpha_p)<0$, then we are in case {\rm (a)}.
\item[\rm{(5)}] If $\alpha_p\in\{\pm u(\beta_1),\ldots,\pm u(\beta_q)\}$, $u^{-1}(\alpha_p)>0$, and $w^{-1}(\alpha_p)>0$, then we are in case {\rm (b)}.
\end{enumerate}
\end{lemma}

\begin{proof}
We will use implicitly some facts discussed in Section \ref{pf1} and earlier in this section. 

Part (1) is clear, as we noted at the beginning of this section that $\{\gamma_1,\ldots,\gamma_q\}$ is a subset of $\{\pm u(\beta_1),\ldots,\pm u(\beta_q)\}$, so if the latter set does not contain $\alpha_p$, then $I_{\alpha_p}(J)\cap[q]$ is empty.  

From now on we assume that $\alpha_p\in\{\pm u(\beta_1),\ldots,\pm u(\beta_q)\}$. Recall the oriented hyperplane $\stackrel{\rightarrow}{H}$ defined in \eqref{hypp}, and the fact that its non-oriented version $H$ is the unique hyperplane orthogonal to $\alpha_p$ which can contain the faces $F_1,\ldots,F_q$ of the gallery $\Pi(J)$. 

If $u^{-1}(\alpha_p)<0$, then $A_{\overline{t}}$ is on the negative side of $\stackrel{\rightarrow}{H}$, by Lemma \ref{lemuw} (1). If $I_{\alpha_p}(J)\cap[q]$ is non-empty, so we are not in case (f), it follows that  the first face $F_i$, with $i\in[q]$, which can be contained in $\stackrel{\rightarrow}{H}$ gives as first symbol of $\overline{S}_p$ either $\mp$ or $-$. In the case of $\mp$, by Lemma \ref{proposition:pm_condition}, a potential second symbol could only be $+$ or $\pm$, which would correspond to a hyperplane orthogonal to $\alpha_p$ different from ${H}$; as this is impossible, we are in case (e). By a completely similar reasoning, if $\overline{S}_p$ starts with $-$, then we are in cases (c) or (d). To conclude, if $u^{-1}(\alpha_p)<0$, then we are in one of the cases {\rm (c)$-$(f)}.

If $w^{-1}(\alpha_p)<0$, then $A_{q}$ is on the positive side of $\stackrel{\rightarrow}{H}$, by Lemma \ref{lemuw} (2). So cases (d)$-$(f) are ruled out, which means that we are in case (c). Similarly, if $w^{-1}(\alpha_p)>0$, then we are in one of the cases {\rm (d)$-$(f)}.

From now assume that $u^{-1}(\alpha_p)>0$, so $A_{\overline{t}}$ is on the positive side of $\stackrel{\rightarrow}{H}$, by Lemma \ref{lemuw} (1). Like above, it follows that, if $\overline{S}_p$ is non-empty, then its first symbol is $\pm$ or $+$. Also like above, based on Lemma \ref{proposition:pm_condition} and on the uniqueness property of the hyperplane $H$ which was also used above, we then deduce that the only possible cases are (a), (b), or (f).

If $w^{-1}(\alpha_p)<0$, then $A_{q}$ is on the positive side of $\stackrel{\rightarrow}{H}$, by Lemma \ref{lemuw} (2), so case (b) is ruled out. Now Lemma \ref{lemma:admissible2} applies, so we deduce that $I_{\alpha_p}(J)\cap[q]$ is non-empty, which means that case (f) is ruled out. Thus, we are in case (a). Similarly, if $w^{-1}(\alpha_p)>0$, then we are in case (b); note that, in this situation, $I_{\alpha_p}(J)\cap[q]\ne\emptyset$ is obvious, as $A_{\overline{t}}$ and $A_q$ are on opposite sides of $\stackrel{\rightarrow}{H}$.  
\end{proof}

We also need to identify the concrete way in which the crystal operator $f_p$ acts on the admissible subset $J$, more precisely, with respect to the $[q]$-window.

\begin{lemma}\label{actfp} The effect of $f_p$ on the $[q]$-window is as follows.
\begin{enumerate} 
\item[\rm{(1)}] In cases {\rm (c)$-$(f)}, the crystal operator $f_p$ does not interact with the $[q]$-window. 
\item[\rm{(2)}] In case {\rm (a)}, if there is interaction, then $f_p$ removes an element of $[q]$ from $J$.
\item[\rm{(3)}] In case {\rm (b)}, if there is interaction, then $f_p$ adds an element of $[q]$ to $J$.
\end{enumerate}
\end{lemma}

\begin{proof}
As case (f) is trivial, we first justify part (1) in cases (c)$-$(e). In these cases, the symbol preceding $\overline{S}_p$ in $S_p$ must be $-$ or $\pm$ (by Lemma \ref{proposition:pm_condition}), which means that the function $g_{\alpha_p}$ has a larger value before the $[q]$-window than everywhere inside it. On the one hand, this implies that the index $m$ in the definition \eqref{eqn:rootF} of $f_p$ cannot be in the $[q]$-window. On another hand, the only way in which the index $k$ in the definition \eqref{eqn:rootF} of $f_p$ can be in the $[q]$-window is if we are in case (d), and the symbol following $\overline{S}_p$ in $S_p$ is $\pm$. But the local maximum of $g_{\alpha_p}$ corresponding to this symbol $\pm$ cannot be the {\em first} occurence of its maximum, due to the previous observation about values before the $[q]$-window. So even in case (d) the index $k$ cannot be in the $[q]$-window.

In cases (a) and (b), the crystal operator $f_p$ can both interact or not interact with the $[q]$-window. Assuming that it does, let us analyze the concrete form of the interaction, based on its definition \eqref{eqn:rootF}. In case (a), $f_p$  is forced to 
remove the folding in the position corresponding to the first $\pm $ in $\overline{S}_p$, and add a
folding in some position $k \in  \left\{ \overline{ 1 } < \ldots < \overline{ t } \right\} $, cf. \eqref{indexset}. 
In case (b), $f_p$ is forced to add a folding
in the position corresponding to the trailing $+$ in $\overline{S}_p$ and possibly remove a folding in
some position $m\in \left\{ \overline{t+1} < \cdots < \overline{n} \right\}$, cf. \eqref{indexset}.
\end{proof}

\subsection{Completion of the proof of Theorem \ref{theorem:comm-f-y}}\label{pf3} At this point we have all the needed ingredients for the proof. We use the same notation as in Sections \ref{pf1} and \ref{pf2}.

\begin{proof}[Proof of Theorem {\rm \ref{theorem:comm-f-y}}.] Let $J':=Y(J)$. First observe that, by definition, $J$ and $J'$ have the same associated Weyl group elements $u$ and $w$. Therefore, by Lemma \ref{critcase}, either cases (d)$-$(f) apply to both $J$ and $J'$, or case (c) applies, or (a), or (b). Furthermore, the sequences of roots $\Gamma(J)$ and $\Gamma(J')$ are identical outside the $[q]$-window, that is, for indices in $I\setminus[q]$. We conclude that the graphs of the functions $g_{\alpha_p}$ corresponding to $J$ and $J'$ are identical outside the $[q]$-window, whereas inside it they differ by an exchange of patterns in cases (d)$-$(f), or in case (c), or (a), or (b). In fact, if we are in cases (a) or (b), then either $f_p$ interacts with the $[q]$-window for both $J$ and $J'$, or for none of them; we call the corresponding cases (a1), respectively (a2), and (b1), respectively (b2). 

Let us denote the Weyl group elements $u$ and $w$ associated to $f_p(J)$ and $f_p(J')$ by $\widetilde{u}$ and $\widetilde{w}$, respectively $\widetilde{u}'$ and $\widetilde{w}'$. We also recall recall the total order \eqref{indexset}, as well as the elements $k$ and $m$ in the definition \eqref{eqn:rootF} of $f_p$. By Lemma \ref{actfp}, $f_p(J)$ is defined if and only if $f_p(J')$ is, and we have the following possibilities (assuming $f_p(J)\ne\Bzero\ne f_p(J')$).
\begin{itemize}
\item If we are in one of the cases (c)$-$(f), or (a2), or (b2), then the pair $k,m$ is the same for $J$ and $J'$, and $k,m\in I\setminus[q]$. Most of the time we have $m\le\overline{t}$ or $k\ge\overline{t+1}$, so $\widetilde{u}=\widetilde{u}'=u$ and $\widetilde{w}=\widetilde{w}'=w$. However, in case (f) we can also have $k\le\overline{t}$ and $m\ge\overline{t+1}$; then, by Remark \ref{changepath}, we have $\widetilde{u}=\widetilde{u}'=s_pu$ and $\widetilde{w}=\widetilde{w}'=s_pw$.
\item In case (a1) the element $k$ is the same for $J$ and $J'$, and $k\le \overline{t}$. Moreover, $f_p$ removes an element from $J\cap[q]$ and one from $J'\cap[q]$. By Remark \ref{changepath}, we have $\widetilde{u}=\widetilde{u}'=s_pu$ and $\widetilde{w}=\widetilde{w}'=w$.
\item In case (b1) the element $m$ is the same for $J$ and $J'$, and $m\ge \overline{t+1}$. Moreover, $f_p$ adds an element to $J\cap[q]$ and one to $J'\cap[q]$. By Remark \ref{changepath}, we have $\widetilde{u}=\widetilde{u}'=u$ and $\widetilde{w}=\widetilde{w}'=s_pw$.
\end{itemize}

We conclude that in all cases we have 
\[\widetilde{u}=\widetilde{u}'\,,\;\;\;\;\;\widetilde{w}=\widetilde{w}'\,,\;\;\;\;\;f_p(J)\setminus[q]=f_p(J')\setminus[q]\,.\]
By definition, it follows that $Y(f_p(J))=f_p(J')$. 
\end{proof}

\section{Explicit description of the quantum Yang-Baxter moves}\label{explicitmoves}

We will now give a type by type description of the quantum Yang-Baxter moves, after recalling that the classical ones were described in \cite{lenccg}. We use freely the setup in Section \ref{ybmoves}. The mentioned description is given in terms of the dihedral reflection group corresponding to the root system $\overline{\Phi}$, which is a subgroup of the Weyl group $W$ (with root system $\Phi$). We denote this subgroup by $\overline{W}$.

\subsection{Dihedral subgroups of Weyl groups}\label{dihedralsubgps} For more information on the subgroup $\overline{W}$ of $W$ we refer to \cite{shi}. For instance, we can easily describe all such subgroups (i.e., subsystems $\overline{\Phi}$) not of type $A_1\times A_1$ for the classical root systems and the one of type $G_2$, as follows. 
\begin{itemize}
\item In all types $A_{n-1}$, $B_n$, $C_n$, and $D_n$, we have the type $A_2$ subsystems with simple roots $\{\varepsilon_i-\varepsilon_j,\,\varepsilon_j-\varepsilon_k\}$, where $1\le i<j<k\le n$. 
\item In types $B_n$, $C_n$, and $D_n$, we have the type $A_2$ subsystems with simple roots $\{\varepsilon_i-\varepsilon_j,\,\varepsilon_j+\varepsilon_k\}$, where $1\le i<j<k\le n$, or $1\le i<k<j\le n$, or $1\le k<i<j\le n$.  
\item In type $C_n$, we have the type $C_2$ subsystems with simple roots $\{\varepsilon_i-\varepsilon_j,\,2\varepsilon_j\}$, where $1\le i<j\le n$; similarly for type $B_n$. 
\item In type $G_2$, we have only the two obvious (strict) subsystems of type $A_2$.
\end{itemize}
In types $E$ and $F$, the number of subsystems  not of type $A_1\times A_1$ (of all possible types) is indicated in the table below.

\[\begin{array}{|c|c|c|}
        \hline & & \vspace{-7pt}\\
\mbox{Type of $\Phi$} & \mbox{$\overline{\Phi}$ of type $A_2$} & \mbox{$\overline{\Phi}$ of type $B_2/C_2$}\\
\hline
E_6 & 120 & -\\
\hline
E_7 & 336 & -\\
\hline
E_8 & 1120 & - \\
\hline
F_4 & 32 & 18\\
\hline
\end{array}\]

%\begin{remark}
%Some results in this section and in the previous sections can be extended to arbitrary subsystems $\overline{\Phi}\subset\Phi$, but we will not need this fact here.  
%\end{remark}

In order to complete the description of our setup, we need a result from \cite{bfpmbo}. Given $w\in W$, consider the partial order on the coset $w\overline{W}$ generated by the relations of the Bruhat order on $W$ corresponding to this coset (i.e., $u<us_\alpha$ for $u\in w\overline{W}$, $\alpha\in\overline{\Phi}$, and $\ell(us_\alpha)>\ell(u)$). Note that this order is, in general, different from the partial order induced from the Bruhat order on $W$. With this clarification, we make a slight correction in the statement of \cite{bfpmbo}[Lemma~5.1].

\begin{proposition}\label{br2} {\rm \cite{bfpmbo}} The Bruhat order on $\overline{W}$ (viewed as a dihedral reflection group) is isomorphic to the partial order on $w\overline{W}$ generated by the relations of the Bruhat order on $W$ corresponding to this coset. More precisely, $w\overline{W}$ has a unique minimal element $\lfloor w\rfloor$, and the map $\overline{w}\mapsto \lfloor w\rfloor\overline{w}$ is an isomorphism of posets $\overline{W}$ and $w\overline{W}=\lfloor w\rfloor\overline{W}$. 
\end{proposition}

\begin{remarks}\label{rem2} (1) As opposed to the case of parabolic subgroups, the factorization $w=\lfloor w\rfloor\overline{w}$ is not length additive. 

(2) The last statement of Proposition \ref{br2} can be rephrased as: for any $\alpha\in\overline{\Phi}^+$, we have $\ell(\overline{w})<\ell(\overline{w}s_\alpha)$ if and only if $\ell(\lfloor w\rfloor\overline{w})<\ell(\lfloor w\rfloor\overline{w}s_\alpha)$.
\end{remarks}

We now state a version of Proposition~\ref{br2} for the quantum Bruhat graph; the corresponding structures on $\overline{W}$ and $w\overline{W}$ are no longer isomorphic, but a weaker version of the above statement holds. This result allows us to reduce any Yang-Baxter move to the ones corresponding to rank two quantum Bruhat graphs, which will be then explicitly given. The proof is postponed to Section \ref{secproof}.

\begin{theorem}\label{qbg2} Under the bijection $\overline{w}\mapsto \lfloor w\rfloor\overline{w}$ between $\overline{W}$ and a coset $w\overline{W}$, every edge of the graph on $w\overline{W}$ induced from $\QB(W)$ corresponds to an edge of $\QB(\overline{W})$. In other words, $w \stackrel{\alpha}{\longrightarrow} 
	ws_{\alpha}$ in $\QB(W)$ with $\alpha\in\overline{\Phi}^+$ implies $\overline{w} \stackrel{\alpha}{\longrightarrow} \overline{w}s_{\alpha}$ in $\QB(\overline{W})$.
\end{theorem}

\subsection{Yang-Baxter moves for rank 2 quantum Bruhat graphs}

By Theorem \ref{qbg2}, the map $Y_{u,w}$ used to define the quantum Yang-Baxter moves in \eqref{setup3} depends only on $\overline{u}$ and $\overline{w}$, so we will denote it by 
$Y_{\overline{u},\overline{w}}$. Hence it suffices to focus on the quantum Bruhat graphs for the dihedral Weyl groups. The ones of type $A_2$, $C_2$, and $G_2$ are shown in Figure \ref{fig:qbg}, where the edge labels correspond to the reflection ordering \eqref{eqn:reflorder}, and $s_1:=s_{\beta_1}$, $s_2:=s_{\beta_q}$. The graph of type $B_2$ is identical with the one of type $C_2$ if we set $\beta_1=\varepsilon_2$ (short root, like in type $C_2$), $\beta_4=\varepsilon_1-\varepsilon_2$ (long root), and $s_1:=s_{\beta_1}$, $s_2:=s_{\beta_4}$, as above. We implicitly use these conventions below. 

\begin{figure}[ht]
    \subfloat[Type $A_2$\label{fig:A2qbg}]{\includegraphics[scale=.45]{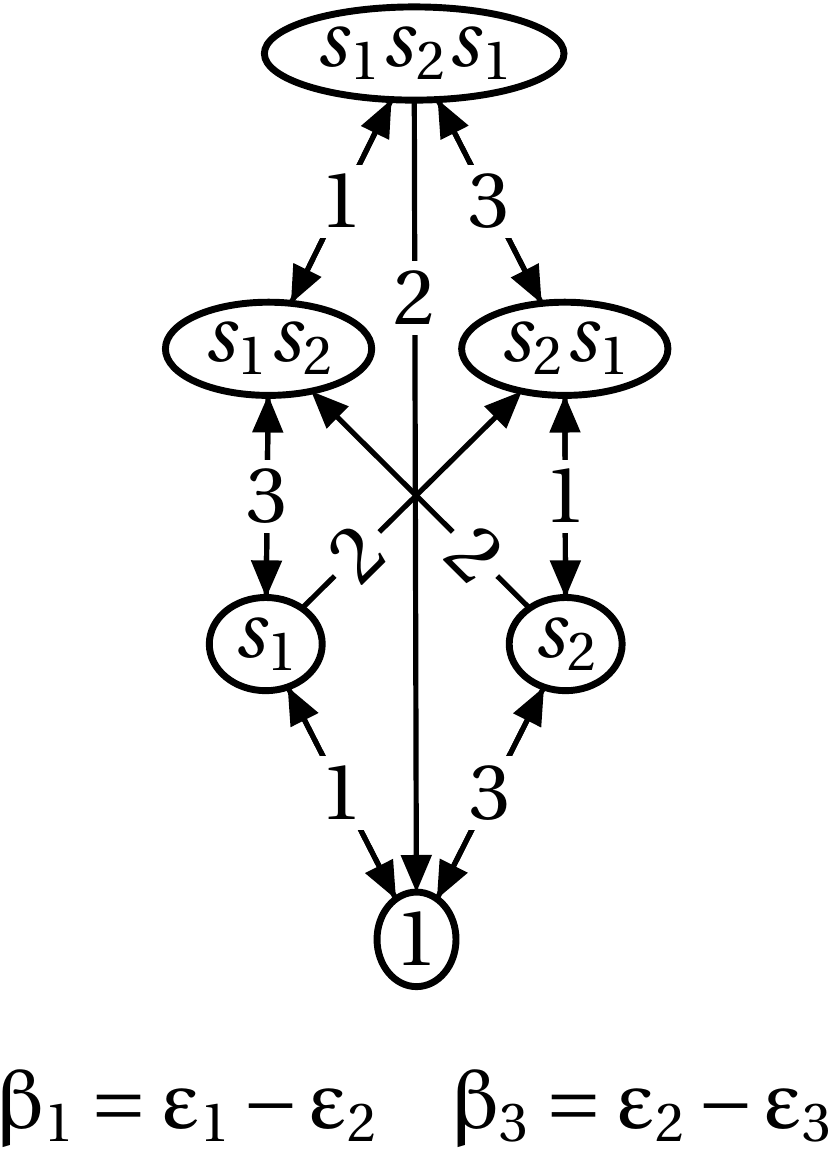}}
    \quad \quad
    \subfloat[Type $C_2$ \label{fig:C2qbg}]{\includegraphics[scale=.45]{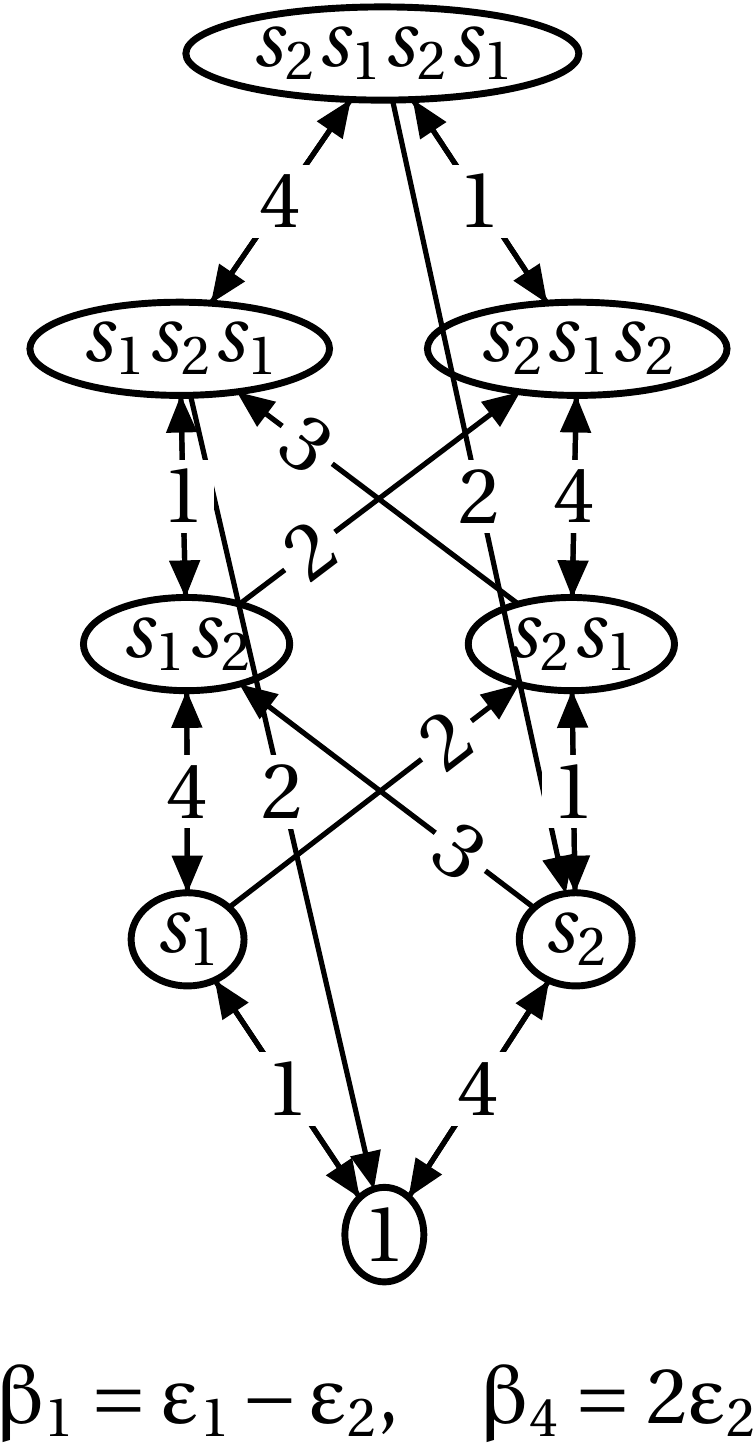}}
    \quad \quad
    %\subfloat[Type $B_2$\label{fig:B2qbg}]{\includegraphics[scale=.4]{B2qbg.pdf}}
    \subfloat[Type $G_2$\label{fig:G2qbg}]{\includegraphics[scale=.45]{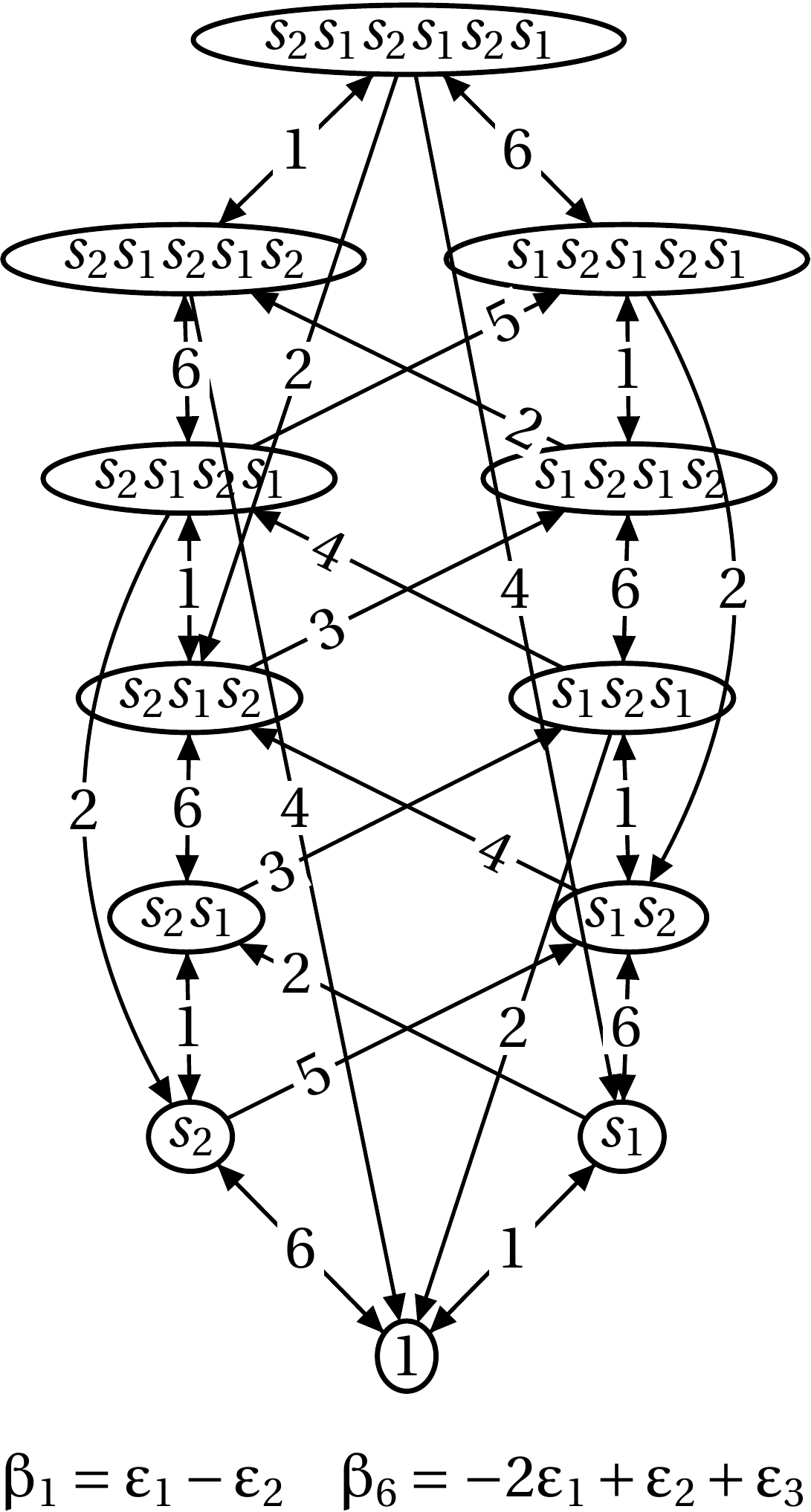}}

    { \hspace{-30pt} \footnotesize Also type $B_2$ with \\
      \hspace{-30pt}$\beta_1 = \varepsilon_2, \;  \beta_4 = \varepsilon_1 - \varepsilon_2$}
    \caption{Quantum Bruhat graphs.}
    \label{fig:qbg}
\end{figure}

Let us now recall from \cite{lenccg} the explicit description of the classical Yang-Baxter moves, i.e., of the map $J\rightarrow Y_{\overline{u},\overline{w}}(J)$ on subsets of $[q]$, for $\overline{u}<\overline{w}$ in $\overline{W}$. Recall that $J$ and $Y_{\overline{u},\overline{w}}(J)$ correspond to saturated chains in Bruhat order from $\overline{u}$ to $\overline{w}$, whose edge labels increase, resp. decrease, with respect to the reflection ordering \eqref{eqn:reflorder}; in fact, as explained in Section \ref{ybmoves}, an index $i$ in $J$ corresponds to the root/edge label $\beta_i$, whereas an index $i$ in $Y_{\overline{u},\overline{w}}(J)$ corresponds to the root $\beta_i':=\beta_{q+1-i}$, cf. \eqref{yblambdachain}. The classical moves can be described in a uniform way (i.e., for all types), and only in terms of $a:=\ell(\overline{u})$ and $b:=\ell(\overline{w})$, as shown in Figure \ref{fig:YB}.

\begin{figure}[h]
\[\begin{array}{ll}
\!\!\!\!\!\!\!\!\!\!\!\mbox{\bfseries Case 0:} &\!\!\!\mbox{$\emptyset\leftrightarrow\emptyset$ if $a=b$}\,.\\[0.05in] 
%\item[
\!\!\!\!\!\!\!\!\!\!\!\mbox{\bfseries Case 1.1:} &\!\!\!\mbox{$\{1\}\leftrightarrow\{q\}$ if $0\le a=b-1\le q-1$}\,.\\[0.05in]
%\item[
\!\!\!\!\!\!\!\!\!\!\!\mbox{\bfseries Case 1.2:} &\!\!\!\mbox{$\{q-a\}\leftrightarrow\{a+1\}$ if $0<a=b-1<q-1$}\,.\\[0.05in]
%\item[
\!\!\!\!\!\!\!\!\!\!\!\mbox{\bfseries Case 2.1:} &\!\!\!\mbox{$\{1,a+2,a+3,\ldots,b\}\leftrightarrow\{a+1,a+2,\ldots,b-1,q\}$ if $0\le a<a+2\le b<q$}\,.\\[0.05in]
%\item[
\!\!\!\!\!\!\!\!\!\!\!\mbox{\bfseries Case 2.2:} &\!\!\!\mbox{$\{1,a+2,a+3,\ldots,b-1,q\}\leftrightarrow\{a+1,a+2,\ldots,b\}$ if $0< a<a+2\le b\le q$}\,.\\[0.05in]
%\item[
\!\!\!\!\!\!\!\!\!\!\!\mbox{\bfseries Case 3:} &\!\!\!\mbox{$[q]\leftrightarrow[q]$ if $a=0$ and $b=q$}\,.
\end{array}\]
\caption{The classical Yang-Baxter moves.}
\label{fig:YB}
\end{figure}
 
The non-classical quantum Yang-Baxter moves $J\rightarrow Y_{\overline{u},\overline{w}}(J)$ are listed in Figure \ref{fig:qYB}, by the corresponding starting vertex $\overline{u}$.
The top and bottom rows are the sequences $J$ and $Y_{\overline{u},\overline{w}}(J)$, respectively, while we also have to consider the reverse moves $Y_{\overline{u},\overline{w}}(J)\rightarrow J$; the mentioned sequences correspond to paths in $\QB(\overline{W})$, as explained in Section \ref{ybmoves} (see also the convention related to the edge labels, which was recalled above, in connection to the classical moves). The symbol $*$ is used to indicate the down edges in $\QB(\overline{W})$. %In type $B_2$, the moves are obtained from the ones for type $C_2$ as follows: first interchange $s_1$ and $s_2$ in $ \overline{u} $, then interchange the two rows following $ \overline{u}$, and read as before.

\subsection{The proof of Theorem {\rm \ref{qbg2}}}\label{secproof}

We start with an alternative description to \eqref{downqbg} of down edges $w\qstep w_\alpha$ in the quantum Bruhat graph, which uses the notion of a {\em quantum root}. This is a root $\alpha \in \Phi^+$ such that $\ell(s_\alpha)=2\htroot(\alpha^\vee)-1$; the definition is motivated by \cite[Lemma 4.3]{bfpmbo}, which states that  any root $\alpha \in \Phi^+$ satisfies $\ell(s_\alpha)\le 2\htroot({\alpha^\vee})-1$. Based on this concept, we can replace the condition in \eqref{downqbg} with the following one:
\begin{equation} \ell(ws_\alpha)=\ell(w)-\ell(s_\alpha)\;\;\mbox{ and }\;\;\mbox{$\alpha$ is a quantum root}\,.
\end{equation}

A characterization of quantum roots is given below.

\begin{lemma}  \label{quantumroot} {\rm \cite{BMO} }
$\alpha\in \Phi^+$ is a quantum root if and only if
\begin{enumerate}
\item[{\rm (1)}] $\alpha$ is a long root, or
\item[{\rm (2)}] $\alpha$ is a short root, and writing $\alpha = \sum_i c_i \alpha_i$, we have
$c_i=0$ for all $i$ such that $\alpha_i$ is long.
\end{enumerate}
Here for simply-laced root systems we consider all roots to be long.
\end{lemma}

\begin{proof}[Proof of Theorem {\rm \ref{qbg2}}] For $u$ in $W$, let $\Phi(u):=\{\beta>0\::\:u(\beta)<0\}$, so $\ell(u)=|\Phi(u)|$. Similarly, letting  $\overline{\Phi}(\overline{u}):=\Phi(\overline{u})\cap\overline{\Phi}$ for $\overline{u}\in\overline{W}$, the length function $\overline{\ell}(\,\cdot\,)$ of $\overline{W}$ (as a dihedral reflection group) is given by $\overline{\ell}(\overline{u})=|\overline{\Phi}(\overline{u})|$. Using the above notation, we will implicitly use the fact that, if $\alpha\in\overline{\Phi}$, then $\overline{us_\alpha}=\overline{u}s_\alpha$; indeed, we have
\[\lfloor u\rfloor \overline{u}s_\alpha=u s_\alpha=\lfloor us_\alpha\rfloor \overline{us_\alpha}=\lfloor u\rfloor \overline{us_\alpha}\,.\]
We will also implicitly use the well-known fact that, given $\alpha\in\Phi^+$, we have $u<us_\alpha$ if and only if $u(\alpha)>0$.

For the remainder of the proof, fix $w\in W$ and $\alpha\in\overline{\Phi}^+$ such that $w<ws_\alpha$, so $w(\alpha)>0$; this is not necessarily a Bruhat cover in $W$. By Proposition~\ref{br2}, we also have $\overline{w}<\overline{w}s_\alpha$. 

{\em Step $1$.} We start by constructing an injection $\iota\,:\,\Phi(w)\hookrightarrow \Phi(ws_\alpha)$, as follows:
\begin{equation}\iota(\beta):=\label{constrmap}\casetwoex{s_\alpha(\beta)}{s_\alpha(\beta)>0}{\beta}{s_\alpha(\beta)<0}\end{equation}
This is likely a folklore result; we only found a partially related fact in the literature, namely \cite{babccg}[Chapter 1, Exercise 12].

We first check that $\iota(\beta)$ is in $\Phi(ws_\alpha)$. Indeed, in the first case we have $ws_\alpha(s_\alpha(\beta))=w(\beta)<0$, by assumptions. In the second case, we have 
\[ws_\alpha(\beta)=w(\beta)-\langle\beta,\alpha^\vee\rangle\,w(\alpha)\,,\]
where $w(\beta)<0$ and $w(\alpha)>0$, from assumptions; we conclude that $ws_\alpha(\beta)<0$ by noting that $\langle\beta,\alpha^\vee\rangle>0$, because otherwise $s_\alpha(\beta)=\beta-\langle\beta,\alpha^\vee\rangle\alpha$ is a positive root, contrary to the assumptions. Injectivity is checked by assuming that $s_\alpha(\beta)=\gamma$ with $s_\alpha(\gamma)<0$, which implies that $s_\alpha(\gamma)=\beta>0$, a contradiction.

Note that $\alpha\in\Phi(ws_\alpha)\setminus\iota(\Phi(w))$. 

{\em Step $2$.} It is not hard to show that, for any $u\in W$, we have 
\begin{equation}\label{intphib}\overline{\Phi}(\overline{u})=\Phi(u)\cap\overline{\Phi}\,.
\end{equation} 
Indeed, given $\beta\in\overline{\Phi}^+$, we have $\overline{u}(\beta)<0$ if and only if $\lfloor u \rfloor\overline{u}(\beta)<0$, by Proposition~\ref{br2}, cf. also Remark~\ref{rem2}~(2). %It means that we also have $\overline{\Phi}(\overline{w}s_\alpha)=\Phi(w s_\alpha)\cap\overline{\Phi}$.

It is also not hard to see that the map $\iota$ restricts to an injection $\overline{\Phi}(\overline{w})\hookrightarrow\overline{\Phi}(\overline{w}s_\alpha)$, as well as to an injection between the complements of the mentioned sets in $\Phi(w)$ and $\Phi(ws_\alpha)$, respectively. Indeed, based on the definition \eqref{constrmap} and \eqref{intphib} for $u=w$ and $u=ws_\alpha$, we only need the obvious fact that $\beta\in\overline{\Phi}$ if and only if $s_\alpha(\beta)\in\overline{\Phi}$.

{\em Step $3$.} Assume that $w \stackrel{\alpha}{\longrightarrow} ws_{\alpha}$ is an edge in $\QB(W)$, i.e., $\ell(ws_\alpha)=\ell(w)+1$. By Step~1, the map $\iota$ is a bijection between $\Phi(w)$ and $\Phi(ws_\alpha)\setminus\{\alpha\}$. By Step~2, it restricts to a bijection between $\overline{\Phi}(\overline{w})$ and $\overline{\Phi}(\overline{w}s_\alpha)\setminus\{\alpha\}$. This implies that $\overline{\ell}(\overline{w}s_\alpha)=\overline{\ell}(\overline{w})+1$, so $\overline{w}\stackrel{\alpha}{\longrightarrow} \overline{w}s_{\alpha}$ is an edge in $\QB(\overline{W})$.

{\em Step $4$.} Now assume that $ws_\alpha \stackrel{\alpha}{\longrightarrow} w$ is an edge in $\QB(W)$, i.e., $\ell(w)=\ell(ws_\alpha)-\ell(s_\alpha)$ and $\alpha$ is a quantum root in $\Phi$. Showing that $\overline{w}s_\alpha \stackrel{\alpha}{\longrightarrow} \overline{w}$ is an edge in $\QB(\overline{W})$ amounts to checking that $\overline{\ell}(\overline{w})=\overline{\ell}(\overline{w}s_\alpha)-\overline{\ell}(s_\alpha)$ and $\alpha$ is a quantum root in $\overline{\Phi}$.

{\em Step $4.1$.} Letting $\Phi_\alpha(w):=\Phi(w)\setminus\Phi(s_\alpha)$, we first note that 
\begin{equation}\label{subsetsa}\Phi(ws_\alpha)\setminus\iota(\Phi_\alpha(w))\subseteq \Phi(s_\alpha)\,,\end{equation}
cf. \cite{babccg}[Chapter 1, Exercise 12]. Indeed, if $\gamma\in\Phi(ws_\alpha)$ is not in $\Phi(s_\alpha)$, then $\beta:=s_\alpha(\gamma)>0$ is in $\Phi_\alpha(w)$ and $\iota(\beta)=\gamma$. It follows that the condition $\ell(w)=\ell(ws_\alpha)-\ell(s_\alpha)$ is equivalent to having equality in \eqref{subsetsa}, that is,
\[\Phi(ws_\alpha)=\iota(\Phi(w))\sqcup\Phi(s_\alpha)\,,\]
cf. \cite{babccg}[Chapter 1, Exercise 13]; here $\sqcup$ denotes disjoint union. Intersecting both sides with $\overline{\Phi}$ and using facts from Step 2, we obtain
\[\overline{\Phi}(\overline{w}s_\alpha)=\iota(\overline{\Phi}(\overline{w}))\sqcup\overline{\Phi}(s_\alpha)\,.\]
As we have seen above, this is equivalent to $\overline{\ell}(\overline{w})=\overline{\ell}(\overline{w}s_\alpha)-\overline{\ell}(s_\alpha)$.

{\em Step $4.2$.} In order to check that the root $\alpha$ is a quantum root in $\overline{\Phi}$ (provided that it is a quantum root in $\Phi\supset\overline{\Phi}$), it suffices to consider $\overline{\Phi}$ of type $B_2$. So $\overline{\Phi}^+=\{\alpha_1,\,\alpha_2,\,\alpha_1+\alpha_2,\,2\alpha_1+\alpha_2\}$, where $\alpha_1$ and $\alpha_2$ are the short and long simple roots, respectively. The only non-quantum (short) root is $\alpha_1+\alpha_2$, so it suffices to check that it cannot be a quantum root in $\Phi$. This follows from the easily checked fact that, in any root system, the expansion (in terms of simple roots) of a long root must contain a long simple root; indeed, we apply this fact to $\alpha_2$ to deduce that the expansion of $\alpha_1+\alpha_2$ (in terms of the simple roots of $\Phi$) must contain a long simple root. 
\end{proof}

\begin{figure}[h!]
{%grouped since we get rid of indentation here
\footnotesize
\ttfamily
\setlength\parindent{0cm}

\begin{minipage}{.39\textwidth}
\underline{Type $A_2$}\newline

$s_1$\newline
\{1*\},\{1*,3 \}\newline
\{3*\},\{2 ,3*\}

$s_1s_2$\newline
\{1 ,2*\},\{3*\},\{1 ,2*,3 \},\{1 ,3*\}\newline
\{1*,3*\},\{1*\},\{1*,2 ,3*\},\{1*,2 \}

$s_1s_2s_1$\newline
\{2*\},\{1*,3*\},\{1*\},\{2*,3 \},\{3*\}\newline
\{2*\},\{2*,3 \},\{3*\},\{1*,3*\},\{1*\} 

$s_2$\newline
\{3*\},\{2 ,3*\}\newline
\{1*\},\{1*,3 \}

$s_2s_1$\newline
\{1*,3*\},\{1*,2 ,3*\},\{1*,2 \},\{1*\}\newline
\{1 ,2*\},\{1 ,2*,3 \},\{1 ,3*\},\{3*\}
\end{minipage}%
\begin{minipage}{.65\textwidth}
\underline{Type $C_2$}\newline

$s_1$\newline
\{1*\},\{1*,4 \}\newline
\{4*\},\{3 ,4*\}

$s_1s_2$\newline
\{1 ,2*\},\{4*\},\{1 ,2*,4 \},\{2 ,4*\}\newline
\{1*,4*\},\{1*\},\{1*,3 ,4*\},\{1*,3 \}

$s_1s_2s_1$\newline
\{2*\},\{1*,4*\},\{1*\},\{2*,4 \},\{1*,2 ,4*\},\{1*,2 \}\newline
\{3*\},\{3*,4 \},\{4*\},\{1 ,3*\},\{1 ,3*,4 \},\{1 ,4*\}

$s_2$\newline
\{4*\},\{3 ,4*\}\newline
\{1*\},\{1*,4 \}

$s_2s_1$\newline
\{1*,4*\},\{1*,3 ,4*\},\{1*,3 \},\{1*\}\newline
\{2 ,3*\},\{2 ,3*,4 \},\{2 ,4*\},\{4*\}

$s_2s_1s_2$\newline
\{1 ,2*,4*\},\{1 ,2*,3 ,4*\},\{1 ,2*,3 \},\{1 ,4*\},\{1 ,2*\},\{4*\}\newline
\{1*,2 ,3*\},\{1*,2 ,3*,4 \},\{1*,2 ,4*\},\{1*,2 \},\{1*,4*\},\{1*\}

$s_2s_1s_2s_1$\newline
\{2*,4*\},\{2*,3 ,4*\},\{2*,3 \},\{4*\},\{2*\},\{1*,4*\},\{1*\}\newline
\{1*,3*\},\{1*,3*,4 \},\{1*,4*\},\{1*\},\{3*\},\{3*,4 \},\{4*\}
\end{minipage}

\begin{minipage}{\textwidth}
\underline{Type $G_2$}\newline

$s_1$\newline
\{1*\},\{1*,6 \}\newline
\{6*\},\{5 ,6*\}

$s_1s_2$\newline
\{1 ,2*\},\{6*\},\{1 ,2*,6 \},\{4 ,6*\}\newline
\{1*,6*\},\{1*\},\{1*,5 ,6*\},\{1*,5 \}

$s_1s_2s_1$\newline
\{2*\},\{1*,6*\},\{1*\},\{2*,6 \},\{1*,4 ,6*\},\{1*,4 \}\newline
\{5*\},\{5*,6 \},\{6*\},\{3 ,5*\},\{3 ,5*,6 \},\{3 ,6*\}

$s_1s_2s_1s_2$\newline
\{2 ,4*\},\{1 ,2*,6*\},\{1 ,2*\},\{6*\},\{2 ,4*,6 \},\{1 ,2*,4 ,6*\},\{1 ,2*,4 \},\{2 ,6*\}\newline
\{1*,5*\},\{1*,5*,6 \},\{1*,6*\},\{1*\},\{1*,3 ,5*\},\{1*,3 ,5*,6 \},\{1*,3 ,6*\},\{1*,3 \}

$s_1s_2s_1s_2s_1$\newline
\{1*,2 ,4*\},\{2*,6*\},\{2*\},\{1*,6*\},\{1*\},\{1*,2 ,4*,6 \},\{2*,4 ,6*\},\{2*,4 \},\{1*,2 ,6*\},\{1*,2 \}\newline
\{1 ,3*,6*\},\{1 ,3*\},\{5*\},\{5*,6 \},\{6*\},\{1 ,3*,5 ,6*\},\{1 ,3*,5 \},\{1 ,5*\},\{1 ,5*,6 \},\{1 ,6*\}

$s_2$\newline
\{6*\},\{5 ,6*\} \newline
\{1*\},\{1*,6 \}

$s_2s_1$\newline
\{1*,6*\},\{1*,5 ,6*\},\{1*,5 \},\{1*\}\newline
\{4 ,5*\},\{4 ,5*,6 \},\{4 ,6*\},\{6*\}

$s_2s_1s_2$\newline
\{1 ,2*,6*\},\{1 ,2*,5 ,6*\},\{1 ,2*,5 \},\{3 ,6*\},\{1 ,2*\},\{6*\}\newline
\{1*,4 ,5*\},\{1*,4 ,5*,6 \},\{1*,4 ,6*\},\{1*,4 \},\{1*,6*\},\{1*\}

$s_2s_1s_2s_1$\newline
\{2*,6*\},\{2*,5 ,6*\},\{2*,5 \},\{1*,3 ,6*\},\{1*,3 \},\{2*\},\{1*,6*\},\{1*\} \newline
\{1 ,3*\},\{1 ,3*,6 \},\{2 ,5*\},\{2 ,5*,6 \},\{2 ,6*\},\{5*\},\{5*,6 \},\{6*\}

$s_2s_1s_2s_1s_2$\newline
\{4*\},\{1 ,4*\},\{1 ,4*,6 \},\{1 ,2*,3 ,6*\},\{1 ,2*,3 \},\{1 ,6*\},\{4*,6 \},\{1 ,2*,6*\},\{1 ,2*\},\{6*\}\newline
\{3*\},\{3*,6 \},\{1*,2 ,5*\},\{1*,2 ,5*,6 \},\{1*,2 ,6*\},\{1*,2 \},\{1*,5*\},\{1*,5*,6 \},\{1*,6*\},\{1*\}

$s_2s_1s_2s_1s_2s_1$\newline
\{1*,4*\},\{4*\},\{4*,6 \},\{2*,3 ,6*\},\{2*,3 \},\{6*\},\{1*,4*,6 \},\{2*,6*\},\{2*\},\{1*,6*\},\{1*\}\newline
\{3*,6*\},\{3*\},\{1*,5*\},\{1*,5*,6 \},\{1*,6*\},\{1*\},\{3*,5 ,6*\},\{3*,5 \},\{5*\},\{5*,6 \},\{6*\}

\end{minipage}
}
\caption{The non-classical Yang-Baxter moves.}
\label{fig:qYB}
\end{figure}

\bibliographystyle{alpha}

\begin{thebibliography}{LNS{\etalchar{+}}13b}

\bibitem[BB05]{babccg}
A. Bj\"{o}rner and F. Brenti.
\newblock \textit{Combinatorics of {C}oxeter groups.}
\newblock Graduate Texts in Mathematics Vol.~231. New York: Springer, 2005.

\bibitem[BMO11]{BMO}
A. Braverman, D.~Maulik, and A.~Okounkov.
\newblock Quantum cohomology of the Springer resolution.''
\newblock {\em Adv. Math.} 227421--458, 2011.

\bibitem[BFP99]{bfpmbo}
F.~Brenti, S.~Fomin, and A.~Postnikov.
\newblock Mixed {B}ruhat operators and {Y}ang-{B}axter equations for {W}eyl
  groups.
\newblock {\em Int. Math. Res. Not.}, 8:419--441, 1999.

\bibitem[BL]{BL}
C. Briggs and C.~Lenart.
\newblock A charge statistic in type $B$.
\newblock in preparation.

\bibitem[D93]{dyehas}
M. Dyer.
\newblock Hecke algebras and shellings of {B}ruhat intervals.
\newblock \textit{Compositio Math.}, 89:91--115, 1993.

\bibitem[FL06]{faltps}
G.~Fourier and P.~Littelmann.
\newblock Tensor product structure of affine {D}emazure modules and limit
  constructions.
\newblock {\em Nagoya Math. J.}, 182:171--198, 2006.

\bibitem[FOS09]{foskr}
G.~Fourier, M.~Okado, and A.~Schilling.
\newblock Kirillov-{R}eshetikhin crystals for nonexceptional types.
\newblock {\em Adv. Math.}, 222:1080--1116, 2009.

\bibitem[FSS07]{fssdsi}
G.~Fourier, A.~Schilling, and M.~Shimozono.
\newblock Demazure structure inside {K}irillov-{R}eshetikhin crystals.
\newblock {\em J. Algebra}, 309:386--404, 2007.

\bibitem[Ful97]{fulyt}
W.~Fulton.
\newblock {\em {Y}oung {T}ableaux}.
\newblock Cambridge University Press, 1997.

\bibitem[FW04]{fawqps}
W.~Fulton and C.~Woodward.
\newblock On the quantum product of {S}chubert classes.
\newblock {\em J. Algebraic Geom.}, 13:641--661, 2004.

\bibitem[GL05]{gallsg}
S.~Gaussent and P.~Littelmann.
\newblock {LS}-galleries, the path model and {MV}-cycles.
\newblock {\em Duke Math. J.}, 127:35--88, 2005.

\bibitem[HKO{\etalchar{+}}99]{hkorff}
G.~Hatayama, A.~Kuniba, M.~Okado, T.~Takagi, and Y.~Yamada.
\newblock Remarks on fermionic formula.
\newblock In {\em Recent developments in quantum affine algebras and related
  topics ({R}aleigh, {NC}, 1998)}, volume 248 of {\em Contemp. Math.}, pages
  243--291. Amer. Math. Soc., Providence, RI, 1999.

\bibitem[HK00]{hkqgcb}
J.~Hong and S.J. Kang.
\newblock {\em Introduction to {Q}uantum {G}roups and {C}rystal {B}ases}, volume~42 of
  {\em Graduate Studies in Mathematics}.
\newblock Amer. Math. Soc., 2000.

%\bibitem[Hum90]{humrgc}
%J.~E. Humphreys.
%\newblock {\em Reflection {G}roups and {C}oxeter {G}roups}, volume~29.
%\newblock Cambridge University Press, Cambridge, 1990.

\bibitem[Kas91]{kascbq}
M.~Kashiwara.
\newblock On crystal bases of the {$q$}-analogue of universal enveloping
  algebras.
\newblock {\em Duke Math. J.}, 63:465--516, 1991.

\bibitem[KN94]{kancgr}
\newblock M.~Kashiwara and T.~Nakashima.
\newblock Crystal graphs for representations of the {$q$}-analogue of classical
  {L}ie algebras.
\newblock {\em J. Algebra}, 165:295--345, 1994.

\bibitem[KR90]{karrym}
A.~Kirillov and N.~Reshetikhin.
\newblock Representations of {Y}angians and multiplicities of the inclusion of
  the irreducible components of the tensor product of representations of simple
  {L}ie algebras.
\newblock {\em J. Sov. Math.}, 52:3156--3164, 1990.

\bibitem[LS79]{lassuc}
A.~Lascoux and M.-P. Sch\mbox{\"{u}}tzenberger.
\newblock Sur une conjecture de {H}. {O}. {F}oulkes.
\newblock {\em C. R. Acad. Sci. Paris \mbox{S\'e}r. I Math.}, 288:95--98, 1979.

\bibitem[Lec02]{lecsc}
C.~Lecouvey.
\newblock Schensted-type correspondence, plactic monoid, and jeu de taquin for
  type {$C\sb n$}.
\newblock {\em J. Algebra}, 247:295--331, 2002.

\bibitem[Lec03]{lecst}
C.~Lecouvey.
\newblock Schensted-type correspondences and plactic monoids for types {$B\sb n$} and {$D\sb n$}.
\newblock {\em J. Algebraic Combin.}, 18:99--133, 2003.

\bibitem[Len07]{lenccg}
C.~Lenart.
\newblock On the combinatorics of crystal graphs, {I}. {L}usztig's involution.
\newblock {\em Adv. Math.}, 211:324--340, 2007.

\bibitem[Len12]{Lenart}
C.~Lenart.
\newblock From {M}acdonald polynomials to a charge statistic beyond type {$A$}.
\newblock {\em J. Combin. Theory Ser. A}, 119:683--712, 2012.

\bibitem[LL11]{lalgam}
C.~Lenart and A.~Lubovsky.
\newblock {A generalization of the alcove model and its applications}.
\newblock {\em J. Algebraic Combin.}, 2014, {\tt DOI:10.1007/s10801-014-0552-3}.
%\newblock {\ttfamily arXiv:1112.2216}.
%\newblock Extended abstract in {\em 24th International Conference on Formal Power Series and
%  Algebraic Combinatorics (FPSAC 2012)}, Discrete Math. Theor. Comput. Sci.
%  Proc. AR, pages 875--886, Nagoya, Japan, 2012.

\bibitem[LNS{\etalchar{+}}12]{unialcmod}
C.~Lenart, S.~Naito, D.~Sagaki, A.~Schilling, and M.~Shimozono.
\newblock A uniform model for {K}irillov-{R}eshetikhin crystals {I}: {L}ifting
  the parabolic quantum {B}ruhat graph.
	\newblock {\em Int. Math. Res. Not.}, 2014, {\tt DOI 10.1093/imrn/rnt263}.

\bibitem[LNS{\etalchar{+}}13a]{lnseda}
C.~Lenart, S.~Naito, D.~Sagaki, A.~Schilling, and M.~Shimozono.
\newblock Explicit description of the action of root operators on quantum
  {L}akshmibai-{S}eshadri paths.
\newblock {\tt arXiv:1308.3529}, 2013.
\newblock To appear in {\em Proceedings of the 5th Mathematical Society of
  Japan Seasonal Institute. Schubert Calculus}, Osaka, Japan, 2012.


\bibitem[LNS{\etalchar{+}}13b]{unialcmod2}
C.~Lenart, S.~Naito, D.~Sagaki, A.~Schilling, and M.~Shimozono.
\newblock A uniform model for {K}irillov-{R}eshetikhin crystals {II}: {P}ath
  models and {$P=X$}.
	\newblock {\ttfamily arXiv:1402.2203}.
%\newblock Extended abstract in {\em 25th International Conference on Formal Power Series and
%  Algebraic Combinatorics (FPSAC 2013)}, Discrete Math. Theor. Comput. Sci.
%  Proc. AS, pages 57--68, Paris, France, 2013.
%\newblock {\ttfamily arXiv:1211.6019}.


\bibitem[LP07]{lapawg}
C.~Lenart and A.~Postnikov.
\newblock Affine {W}eyl groups in {$K$}-theory and representation theory.
\newblock {\em Int. Math. Res. Not.}, pages 1--65, 2007.
\newblock Art. ID rnm038.

\bibitem[LP08]{lapcmc}
C.~Lenart and A.~Postnikov.
\newblock A combinatorial model for crystals of {K}ac-{M}oody algebras.
\newblock {\em Trans. Amer. Math. Soc.}, 360:4349--4381, 2008.

\bibitem[LS13]{lascec}
C.~Lenart and A.~Schilling.
\newblock Crystal energy via the charge in types {$A$} and {$C$}.
\newblock {\em Math. Z.}, 273:401--426, 2013.


\bibitem[Lit94]{litlrr}
P.~Littelmann.
\newblock {A Littlewood-Richardson rule for symmetrizable Kac-Moody algebras}.
\newblock {\em Invent. Math.}, 116:329--346, 1994.

\bibitem[Lit95]{litpro}
P.~Littelmann.
\newblock Paths and root operators in representation theory.
\newblock {\em Ann. of Math. {\rm (2)}}, 142:499--525, 1995.

\bibitem[NS08]{NS08}
S. Naito and D. Sagaki. 
\newblock Lakshmibai-Seshadri paths of level-zero weight shape and 
one-dimensional sums associated to level-zero fundamental representations.
\newblock {\it Compos. Math.}, 144:1525--1556, 2008.


\bibitem[Pos05]{posqbg}
A.~Postnikov.
\newblock {Q}uantum {B}ruhat graph and {S}chubert polynomials.
\newblock {\em Proc. Amer. Math. Soc.}, 133:699--799, 2005.

\bibitem[ST12]{satdck}
A.~Schilling and P.~Tingley.
\newblock Demazure crystals, {K}irillov-{R}eshetikhin crystals, and the energy
  function.
\newblock {\em Electron. J. Combin.}, 19:P2, 2012.

\bibitem[Shi93]{shi}
J.-Y. Shi.
\newblock Some numeric results on root systems.
\newblock {\em Pacific J. Math.}, 160:155--164, 1993.

\bibitem[Yam98]{yampfg}
Y.~Yamane.
\newblock Perfect crystals of $U_q(G_2^{(1)})$.
\newblock {\em J. Algebra}, 210:440--486, 1998.


\end{thebibliography}

\newcommand{\etalchar}[1]{$^{#1}$}

\end{document}